\documentclass[12pt, reqno]{amsart}
\makeatletter
\usepackage[margin=2cm]{geometry}

\usepackage{amsmath, amssymb, amsfonts, amstext, verbatim, amsthm, mathrsfs}
\usepackage[mathcal]{eucal}
\usepackage{microtype}
\usepackage[all]{xy}

\usepackage[usenames]{color}
\usepackage{aliascnt} 
\usepackage[colorlinks=true,linkcolor=blue,citecolor=blue,urlcolor=blue,citebordercolor={0 0 1},urlbordercolor={0 0 1},linkbordercolor={0 0 1}]{hyperref} 
\usepackage{enumerate}
\usepackage{xspace}
\usepackage{tikz}
\theoremstyle{plain}

\newcommand{\refnewtheoremn}[4]{
\newaliascnt{#1}{#2}
\newtheorem{#1}[#1]{#3}
\aliascntresetthe{#1}
\expandafter\providecommand\csname #1autorefname\endcsname{#4}}

\newcommand{\refnewtheorem}[3]{\refnewtheoremn{#1}{#2}{#3}{#3}}

\def\makeCal#1{
\expandafter\newcommand\csname c#1\endcsname{\mathcal{#1}}}
\def\makeBB#1{
\expandafter\newcommand\csname b#1\endcsname{\mathbb{#1}}}
\def\makeFrak#1{
\expandafter\newcommand\csname f#1\endcsname{\mathfrak{#1}}}

\count@=0
\loop
\advance\count@ 1
\edef\y{\@Alph\count@}
\expandafter\makeCal\y
\expandafter\makeBB\y
\expandafter\makeFrak\y
\ifnum\count@<26
\repeat

\newtheorem{thm}{Theorem}[section]

\refnewtheorem{prethm}{thm}{Pre-theorem}

\refnewtheorem{lemma}{thm}{Lemma}
\refnewtheorem{claim}{thm}{Claim}
\refnewtheorem{cor}{thm}{Corollary}
\refnewtheorem{conj}{thm}{Conjecture}
\refnewtheorem{prop}{thm}{Proposition}
\refnewtheorem{quest}{thm}{Question}
\refnewtheorem{amplif}{thm}{Amplification}
\refnewtheorem{scholium}{thm}{Scholium}

\refnewtheorem{tablethm}{thm}{Table}
\refnewtheorem{assumption}{thm}{Assumption}
\refnewtheorem{conclusion}{thm}{Conclusion}
\refnewtheorem{problem}{thm}{Problem}
\refnewtheorem{ansatz}{thm}{Ansatz}

\theoremstyle{definition}
\refnewtheorem{remark}{thm}{Remark}
\refnewtheorem{remarks}{thm}{Remarks}
\refnewtheorem{defn}{thm}{Definition}
\refnewtheorem{notn}{thm}{Notation}
\refnewtheorem{const}{thm}{Construction}
\refnewtheorem{ex}{thm}{Example}
\refnewtheorem{fact}{thm}{Fact}
\refnewtheorem{recollection}{thm}{Recollection}

\newcommand{\fg}{\mathfrak{g}}

\newcommand{\hS}{\Hat{S}}

\renewcommand{\Im}{\operatorname{Im}}
\renewcommand{\Re}{\operatorname{Re}}
\newcommand {\Hom}{\operatorname{Hom}}
\newcommand {\Aut}{\operatorname{Aut}}

\newcommand {\Ext}{\operatorname{Ext}}

\newcommand{\DT}{\operatorname{DT}}
\newcommand{\ch}{\operatorname{ch}}
\newcommand{\Coh}{\operatorname{Coh}}
\newcommand{\GW}{\operatorname{GW}}

\newcommand{\Quad}{\operatorname{Quad}}

\newcommand{\GL}{\operatorname{GL}}
\newcommand{\SU}{\operatorname{SU}}
\renewcommand{\dim}{\operatorname{dim}}
\newcommand{\height}{{H}}

\newcommand{\Li}{\operatorname{Li}}

\newcommand{\GV}{\operatorname{GV}}
\newcommand{\td}{\operatorname{td}}
\newcommand{\interior}{\operatorname{int}}

\newcommand{\into}{\hookrightarrow}

\newcommand {\<}{\langle}
\renewcommand {\>}{\rangle}
\newcommand{\isom}{\cong}
\newcommand{\half}{\frac{1}{2}}
\newcommand{\tensor}{\otimes}

\newcommand{\D}{D}

\newcommand{\dual}{\vee}

\renewcommand{\Phi}{X}
\renewcommand{\Psi}{Y}

\usepackage{calligra}
\DeclareMathAlphabet{\mathcalligra}{T1}{calligra}{m}{n}
\DeclareFontShape{T1}{calligra}{m}{n}{<->s*[2.2]callig15}{}

\newcommand{\rr}{\scriptr}

\renewcommand{\rr}{{r}}

\title{Riemann-Hilbert problems from Donaldson-Thomas theory}
\author{Tom Bridgeland}
\date{}

\begin{document}

\begin{abstract}{We study a class of Riemann-Hilbert problems arising naturally in Donaldson-Thomas theory.   In certain special cases we show that these problems have unique solutions which can be written explicitly as  products of gamma functions. We briefly explain connections with Gromov-Witten theory and exact WKB analysis.}
\end{abstract}

\maketitle



\section{Introduction}

In this paper we study a class of Riemann-Hilbert problems  arising naturally in Donaldson-Thomas theory. They involve maps from the complex plane to an algebraic torus, with prescribed discontinuities along a  collection of rays, and are closely related to the Riemann-Hilbert problems considered by Gaiotto, Moore and Neitzke \cite{GMN1}; in physical terms we are considering the conformal limit of their story. The same problems have also been considered by Stoppa and his collaborators \cite{BSt,FGS}.
One of our main results is  that in the `uncoupled' case the Riemann-Hilbert problem has a unique solution which can be written explicitly using products of  gamma functions (Theorem \ref{one}). The inspiration for this  comes from a calculation of Gaiotto \cite{G}.

 We begin by introducing the notion of a BPS structure. This is   a special case of Kontsevich and Soibelman's notion of a  stability structure \cite{KS1}. In mathematical terms, it describes the output of unrefined Donaldson-Thomas theory  applied to a three-dimensional Calabi-Yau category with a stability condition. There is also a  notion of a variation of BPS structures over a complex manifold, which axiomatises the  behaviour of  Donaldson-Thomas invariants under  changes of stability: the main ingredient  is the Kontsevich--Soibelman wall-crossing formula. 

To any BPS structure satisfying a natural growth condition we associate a Riemann-Hilbert problem. We go to some pains to set this up precisely. We then prove the existence of a unique solution in the  uncoupled case referred to above. 
Given a variation of BPS structures over a complex manifold $M$, and a family of solutions to the corresponding Riemann-Hilbert problems, we can define a piecewise holomorphic function on $M$ which we call the $\tau$-function. In the uncoupled case we give an explicit expression for this function using the Barnes $G$-function (Theorem \ref{doss}).

Variations of BPS  structures  also arise in  theoretical physics in the study of quantum field theories with $N=2$ supersymmetry (see for example \cite{GMN1}). Our $\tau$-function  then seems to be closely related to the partition function of the theory. Thus, as a rough slogan, one can think of the BPS invariants as encoding the Stokes phenomena which arise when Borel resumming the genus expansion of the  free energy. As an example of this relationship,  we compute in Section 8 the asymptotic expansion of $\log(\tau)$ for the variation of BPS structures arising in  topological string theory,  and show that it reproduces the genus 0 part of the Gopakumar-Vafa expression for the Gromov-Witten generating function.  

Another interesting class of BPS structures arise in theoretical physics from supersymmetric gauge theories of class $S$. In the case of gauge group $\SU(2)$ these theories play a central r{o}le in  the paper of Gaiotto, Moore and Neitzke \cite{GMN2}. The corresponding BPS structures depend on a Riemann surface equipped with a meromorphic quadratic differential, and  the  BPS invariants encode counts of finite-length geodesics. 
These structures arise mathematically via the stability conditions studied by the author and Smith \cite{BS}. The work of Iwaki and Nakanishi \cite{IN} shows that the corresponding Riemann-Hilbert problems can be partially solved using the techniques of exact WKB analysis. We expect our $\tau$-function in this case to be closely related to the one computed by topological recursion \cite{ey}.

The theory we attempt to develop here is  purely mathematical. One potential advantage of our approach is its generality: the only input for the theory is a triangulated category satisfying the three-dimensional Calabi-Yau condition. When everything works, the output is   a complex manifold - the space of stability conditions - equipped with an interesting piecewise holomorphic function: the $\tau$-function. Note that the theory is inherently global and non-perturbative: it does not use expansions about some chosen limit point in the space of stability conditions. 

We should admit straight away  that at present there are many unanswered questions and unsolved technical problems with the theory. In  particular, for general BPS structures we have no  existence or uniqueness results for solutions to the Riemann-Hilbert problem. It is also not clear why the $\tau$-function as defined here should exist in the general uncoupled case. Nonetheless, the strong analogy with Stokes structures in the theory of differential equations, and the non-trivial answers obtained here  (see also \cite{tocome}) provide adequate mathematical motivation to study these problems further.

\subsection*{Plan of the paper}

In Section 2 we introduce basic definitions concerning BPS structures. Section 3  contains a summary of the contents of the paper with  technical details deferred to later sections. 
 In Section 4 we discuss the Riemann-Hilbert problem associated to a BPS structure.  In Section 5 we solve this problem    in the uncoupled case using elementary properties of the gamma function. Sections 6 and 7 discuss 
 the connections with Gromov-Witten invariants and exact WKB analysis referred to above. In Appendix A  we give a rigorous definition of a variation of BPS structures following Kontsevich and Soibelman. Appendix B contains some simple analytic results involving partially-defined self-maps of algebraic tori.

 \subsection*{Acknowledgements}
 The author is very grateful to  Kohei Iwaki, Andy Neitzke, Ivan Smith and Bal{\'a}zs Szendr{\H o}i  for extremely helpful discussions during the long gestation period of this paper. This work was partially supported by an ERC Advanced grant. The calculations in Section \ref{gw} were carried out jointly with Kohei Iwaki during a visit to the (appropriately named)  Bernoulli Centre in  Lausanne. I would also like to thank Sven Meinhardt for a careful reading of the preprint version, and for several useful comments.  

\section{BPS structures: initial definitions}

In this section we introduce the abstract notion  of a BPS structure and explain the corresponding picture of active rays and  BPS automorphisms. In mathematics, these structures arise naturally as the output of generalized Donaldson-Thomas theory applied to a  three-dimensional Calabi-Yau triangulated category  with a stability condition. These ideas go back to Kontsevich and Soibelman \cite[Section 2]{KS1}, building on work of Joyce (see \cite{Br4} for a gentle review). The same structures also arise in theoretical physics in the study of  quantum field  theories with $N=2$ supersymmetry \cite[Section 1]{GMN2}.

\subsection{Definition and terminology}

The following definition  is a special case of the notion of stability data on a graded Lie algebra \cite[Section 2.1]{KS1}. It was also studied by Stoppa and his collaborators \cite[Section 3]{BSt}, \cite[Section 2]{FGS}.

\begin{defn}
A BPS structure  consists of \begin{itemize}
\item[(a)] A finite-rank free abelian group $\Gamma\isom \bZ^{\oplus n}$, equipped with a skew-symmetric form \[\<-,-\>\colon \Gamma \times \Gamma \to \bZ,\]

\item[(b)] A homomorphism of abelian groups
$Z\colon \Gamma\to \bC$,\smallskip

\item[(c)] A map of sets
$\Omega\colon \Gamma\to \bQ,$
\end{itemize}

satisfying the following properties:

\begin{itemize}
\item[(i)] Symmetry: $\Omega(-\gamma)=\Omega(\gamma)$ for all $\gamma\in \Gamma$,\smallskip
\item[(ii)] Support property: fixing a norm  $\|\cdot\|$ on the finite-dimensional vector space $\Gamma\tensor_\bZ \bR$, there is a constant $C>0$ such that 
\begin{equation}
\label{support}\Omega(\gamma)\neq 0 \implies |Z(\gamma)|> C\cdot \|\gamma\|.\end{equation}
\end{itemize}
\end{defn}

The lattice $\Gamma$ will be  called the charge lattice, and the form $\<-,-\>$ is the intersection form. The group homomorphism $Z$ is  called the central charge. The rational numbers $\Omega(\gamma)$ are called BPS invariants. A class $\gamma\in \Gamma$ will be called active if $\Omega(\gamma)\neq 0$.

\subsection{Donaldson-Thomas invariants}

The Donaldson-Thomas (DT) invariants  of a BPS structure $(\Gamma,Z,\Omega)$ are defined by the expression
\begin{equation}
\label{bps}\DT(\gamma)=\sum_{ \gamma=m\alpha}   \frac{1}{m^2} \, \Omega(\alpha)\in \bQ,\end{equation}
where the sum is over integers $m>0$ such that $\gamma$ is divisble by $m$ in the lattice $\Gamma$.
The BPS  and DT invariants are equivalent data: we can write
\begin{equation}\Omega(\gamma)=\sum_{ \gamma=m\alpha}   \frac{\mu(m)}{m^2}  \DT(\alpha),\end{equation}
where $\mu(m)$ is the M{\"o}bius function. One reason to prefer the BPS invariants is that in many examples they are known, or conjectured, to be integers. Note however that this depends on  a genericity assumption  (see Definition \ref{gob}),  without which  integrality fails even in very simple examples (see Section \ref{newnew}).  

\subsection{Poisson algebraic torus}

Given a lattice $\Gamma\isom \bZ^{\oplus n}$ equipped with a skew-symmetric form $\<-,-\>$ as above, we  consider 
the algebraic torus
\[\bT_+=\Hom_\bZ(\Gamma,\bC^*)\isom (\bC^*)^n,\]
and its co-ordinate ring (which is also the group ring of the lattice $\Gamma$)
\[\bC[\bT_+]=\bC[\Gamma]\isom \bC[y_1^{\pm 1}, \cdots, y_n^{\pm n}].\]
We write $y_\gamma\in \bC[\bT_+]$ for the  character of $\bT_+$ corresponding to an element $\gamma\in \Gamma$. The skew-symmetric form $\<-,-\>$ induces an invariant Poisson structure on $\bT_+$, given on characters by
\begin{equation}
\label{poisson}\{y_{\alpha}, y_{\beta}\}= \<\alpha,\beta\>\cdot y_{\alpha}\cdot y_{\beta}.\end{equation}
As well as the torus $\bT_+$, it will also be important for us to consider an associated torsor
\[\bT_-= \{g\colon \Gamma \to \bC^*: g(\gamma_1+\gamma_2)=(-1)^{\<\gamma_1,\gamma_2\>} g(\gamma_1)\cdot g(\gamma_2)\},\]
which we call the twisted torus. This space is discussed in more detail in the next subsection, but since the difference between  $\bT_+$ and $\bT_-$  just has the effect of introducing signs into various formulae, it  can safely be ignored at first reading.

\subsection{Twisted torus}
Let us again fix a lattice $\Gamma\isom \bZ^{\oplus n}$ equipped with a skew-symmetric form $\<-,-\>$.
The torus $\bT_+$ acts freely and transitively on the twisted torus $\bT_-$  via
\[(f\cdot g)(\gamma)=f(\gamma)\cdot g(\gamma)\in \bC^*, \quad f\in \bT_+, \quad  g\in \bT_-.\]
Choosing a base-point $g_0\in \bT_-$ therefore gives a bijection \begin{equation}
\label{bloxy}\theta_{g_0}\colon \bT_+\to \bT_-, \qquad f\mapsto f\cdot g_0.\end{equation}
We can use the identification $\theta_{g_0}$ to give $\bT_-$ the structure of an algebraic variety. The result is independent of the choice of base-point $g_0\in \bT_-$, since the translation maps on $\bT_+$ are algebraic.  
Similarly, the Poisson structure \eqref{poisson} on $\bT_+$  is invariant under translations, and hence can be transferred to $\bT_-$ via the map \eqref{bloxy}.

The co-ordinate ring of $\bT_-$ is spanned as a vector space by the functions
\[x_\gamma\colon \bT_-\to \bC^*, \qquad x_\gamma(g)=g(\gamma)\in \bC^*,\]
which we  refer to as twisted characters.  Thus
\begin{equation}\bC[\bT_-]=\bigoplus_{\gamma\in \Gamma} \bC\cdot x_\gamma, \qquad x_{\gamma_1}\cdot x_{\gamma_2}=(-1)^{\<\gamma_1,\gamma_2\>}\cdot x_{\gamma_1+\gamma_2}.\end{equation}
The Poisson bracket on $\bC[\bT_-]$ is given on twisted characters by 
 \begin{equation}
\label{poisson2}\{x_{\alpha}, x_{\beta}\}= \<\alpha,\beta\>\cdot x_{\alpha}\cdot x_{\beta}.\end{equation}
From now on we shall denote the twisted torus $\bT_-$ simply by $\bT$.

\subsection{Ray diagram}
\label{raydiagram}
Let $(\Gamma,Z,\Omega)$ be a BPS structure.  The support property implies that in any bounded region of $\bC$ there are only finitely many points of the form $Z(\gamma)$ with $\gamma\in \Gamma$ an active class. It also implies that all such points are nonzero.
\begin{figure}
\begin{tikzpicture}[scale=0.4]
\draw (5,0) circle [radius=0.1];
\draw[->] (5.1,0.1) -- (8,2);
\draw[->] (4.9,0.1) -- (3,2);
\draw[->](5,0.1)--(5,4);
\draw[->] (10-5.1,-0.1) -- (10-8,-2);
\draw[->] (10-4.9,-0.1) -- (10-3,-2);
\draw[->](10-5,-0.1)--(10-5,-4);
\draw (9.4,2) node { $\scriptstyle Z(\gamma_1)$};
\draw (1.75,2) node {$\scriptstyle Z(\gamma_2)$};
\draw (3.8,4) node {$\scriptstyle Z(\gamma_3)$};
\draw (10-9.4,-2) node { $\scriptstyle -Z(\gamma_1)$};
\draw (10-1.75,-2) node {$\scriptstyle -Z(\gamma_2)$};
\draw (10-3.8,-4) node {$\scriptstyle -Z(\gamma_3)$};
\draw[fill] (8.05,2.05) circle [radius=0.05];
\draw[fill] (2.95,2.05) circle [radius=0.05];
\draw[fill] (4.95,4.05) circle [radius=0.05];
\draw[fill] (10-8.05,-2.05) circle [radius=0.05];
\draw[fill] (10-2.95,-2.05) circle [radius=0.05];
\draw[fill] (10-4.95,-4.05) circle [radius=0.05];
\end{tikzpicture}
\caption{The ray diagram associated to a BPS structure.
\label{fig}}
\end{figure}
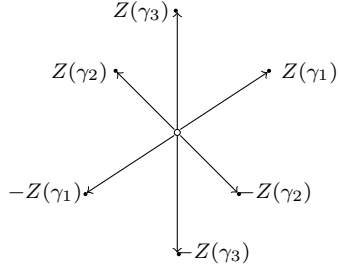

By a ray in $\bC^*$ we mean a subset of the form
$\ell=\bR_{>0}\cdot z$ for some $z\in \bC^*$.
Such a ray  will be called active if it contains a point $Z(\gamma)$ for some active class $\gamma\in \Gamma$.  Taken together the active rays form a  picture as in Figure \ref{fig}, which we call the ray diagram of the BPS structure. In general there will be countably many active rays.
We define the height of an active ray $\ell\subset \bC^*$ to be \[\height(\ell)=\inf\big\{|Z(\gamma)|: \gamma \in \Gamma\text{ such that } Z(\gamma)\in \ell\text{ and } \Omega(\gamma)\neq 0\big\}.\]Non-active rays are considered to have infinite height. The support property  ensures that for any $H>0$ there are only finitely many rays of height $<H$.

Associated to any ray $\ell\subset \bC^*$ is a formal sum of twisted characters
\begin{equation}
\label{hamiltonian}\DT(\ell)=-\hspace{-.8em}\sum_{\gamma\in \Gamma: Z(\gamma)\in \ell} \DT(\gamma) \cdot x_\gamma.\end{equation}
Naively, we would like to view  $\DT(\ell)$  as a well-defined holomorphic function on the twisted torus $\bT$, and consider the associated time 1 Hamiltonian flow as a Poisson automorphism
\[\bS(\ell)\in \Aut(\bT).\]
We refer to this as the BPS automorphism associated to the ray $\ell$;
making good sense of it is one of the main technical problems we shall need to deal with.

\subsection{Further terminology}
\label{defns}

In this section we gather some terminology for describing BPS structures of various special kinds.

\begin{defn}
\label{la}
We say that a BPS structure $(Z,\Gamma,\Omega)$ is
\begin{itemize}
\item[(a)] finite, if there are only finitely many active classes $\gamma\in \Gamma$;

\item[(b)] ray-finite, if for any ray $\ell\subset \bC^*$ there are only finitely many active classes $\gamma\in \Gamma$ for which $Z(\gamma)\in \ell$;

\item[(c)] convergent, if for some $R>0$
\begin{equation}\label{concon}\big.\sum_{\gamma\in \Gamma} |\Omega(\gamma)|\cdot e^{-R|Z(\gamma)|}<\infty.\end{equation}
\end{itemize}
\end{defn}

An equivalent condition to \eqref{concon} already appears in the work of Gaiotto, Moore and Neitzke \cite[Appendix C]{GMN1}. The same condition also plays a prominent role in the work of Barbieri and Stoppa \cite[Definition 3.5]{BSt}.

\begin{defn}
\label{gob}
We say that a BPS structure $(Z,\Gamma,\Omega)$ is
\begin{itemize}
\item[(a)] {uncoupled}, if  for any two active classes $\gamma_1,\gamma_2\in \Gamma$  one has
 $\<\gamma_1,\gamma_2\>=0$;

\item[(b)] generic, if for any two active classes $\gamma_1,\gamma_2\in \Gamma$ one has 
\[\bR_{>0} \cdot Z(\gamma_1)=\bR_{>0} \cdot Z(\gamma_2) \implies \< \gamma_1, \gamma_2\>=0.\]
\item[(c)] integral, if the BPS invariants $\Omega(\gamma)\in \bZ$ are all integers.
\end{itemize}
\end{defn}

The  uncoupled condition ensures that the Hamiltonian flows for any pair of functions on $\bT$ of the form $\DT(\gamma)\cdot x_\gamma$ commute. This situation corresponds to the case of  `mutually local corrections'  in \cite{GMN1}.  Genericity is the weaker condition that all such flows for which $Z(\gamma)$ lies on a given fixed ray $\ell\subset \bC^*$ should commute.

\subsection{BPS automorphisms}
\label{really}

As mentioned above, the main technical problem we have to deal with is making suitable definitions of the BPS automorphisms $\bS(\ell)$ associated to a BPS structure. Since we will use three different approaches at various points in the paper, it is perhaps worth briefly summarising  these here.

\begin{itemize}
\item[(i)] {\it Formal approach.}  If we are only interested in the elements $\bS(\ell)$ for rays $\ell\subset \bC^*$ lying in a fixed  acute sector $\Delta\subset \bC^*$, then we can work with a variant of the algebra $\bC[\bT]$ consisting of formal sums of the form \[\big.\sum_{Z(\gamma)\in \Delta} a_\gamma\cdot x_\gamma,\quad a_\gamma\in \bC,\] such that for any $H>0$ there are only finitely many terms with $|Z(\gamma)|<H$. This is the approach we shall use in Appendix \ref{two} to define variations of BPS structures: it has the advantage of not requiring any extra assumptions.\smallskip

\item[(ii)] {\it Analytic approach.} In Appendix \ref{converge},   we  associate to each convex sector $\Delta\subset \bC^*$, and each real number $R>0$, a non-empty analytic open subset $U_\Delta(R)\subset \bT$ defined to be the interior of the subset 
\[\big\{g\in \bT: Z(\gamma)\in \Delta \text{ and } \Omega(\gamma)\neq 0\implies |g(\gamma)|<\exp(-R\|\gamma\|)\big\}\subset \bT.\]
 We then show that if the BPS structure is convergent, and $R>0$ is sufficiently large, then  for any  active ray $\ell\subset \Delta$, the formal series $\DT(\ell)$ is absolutely convergent on   $U_\Delta(R)\subset \bT$, and that the time 1 Hamiltonian flow of the resulting  function defines a holomorphic  embedding \[\bS(\ell) \colon U_\Delta(R)\rightarrow \bT.\]
 We can then view this map as being a partially-defined automorphism of $\bT$. For a more precise statement see Proposition \ref{propy}.
\smallskip

\item[(iii)] {\it Birational approach.} In  the case of a generic, integral and ray-finite  BPS structure, the partially-defined automorphisms $\bS(\ell)$  discussed in (ii)  extend to birational automorphisms of $\bT$; see Propostion \ref{barneyy}. The induced pullback  of twisted characters is expressed by the formula
\begin{equation}\bS(\ell)^*(x_\beta)=x_\beta\cdot \prod_{Z(\gamma)\in \ell}(1-x_\gamma)^{\,\Omega(\gamma)\<\gamma,\beta\>}\end{equation}
which is often taken as a definition (see e.g. \cite[Section 2.2]{GMN1}). 
\end{itemize}

\subsection{Doubling construction}
\label{double}

It is often useful to be able to assume that the form $\<-,-\>$ is non-degenerate. To reduce to this case we can use the  following doubling construction \cite[Section 2.6]{KS1}.
Suppose given 
a BPS structure $(\Gamma, Z, \Omega)$. The doubled BPS structure takes the form
\[(\Gamma\oplus \Gamma^\dual, Z, \Omega),\]
where $\Gamma^\dual=\Hom_\bZ(\Gamma,\bZ)$ is the dual lattice. We equip the doubled lattice $\Gamma_D=\Gamma\oplus \Gamma^\dual$ with the non-degenerate skew-symmetric form
\begin{equation}
\label{job}\big\<(\gamma_1,\lambda_1), (\gamma_2,\lambda_2)\big\>=\<\gamma_1,\gamma_2\>+\lambda_1(\gamma_2)-\lambda_2(\gamma_1).\end{equation}
The central charge is defined by
$Z(\gamma,\lambda)=Z(\gamma),$
and the BPS invariants by
\[\Omega(\gamma,\lambda)=\begin{cases} \Omega(\gamma) &\text{ if } \lambda= 0, \\ 0&\text{ otherwise}.\end{cases}\]
The support property reduces to that for the original structure.
Slightly more generally, we can consider BPS structures of the form $(\Gamma\oplus \Gamma^\dual, Z\oplus Z^\vee, \Omega)$ where $Z^\vee\colon \Gamma^\vee\to \bC$ is an arbitrary group homomorphism.

\subsection{A basic example: the Kronecker quiver}
\label{newnew}

Interesting examples of BPS structures can be obtained from  Donaldson-Thomas theory.  In the case of the  generalised Kronecker quiver with $k>0$ arrows (see also \cite{GP}) these  BPS structures $(\Gamma,Z,\Omega)$ have \[\Gamma=\bZ e_1\oplus \bZ e_2, \qquad \<e_1,e_2\>=-k,\] and  are specified by a central charge $Z\colon \Gamma\to \bC$ satisfying $\Im Z(e_i)>0$ for $i=1,2$. The  BPS invariants depend only on the sign of \[\nu=\Im \,(Z(e_2)/Z(e_1)).\]

If $\nu>0$ then
the only nonzero BPS invariants are
$\Omega(\pm e_1)=\Omega(\pm e_2)=1.$
In the  case $\nu=0$ the BPS structure is non-generic, and also non-integral in general: Joyce and Song \cite[Section 6.2]{JS} show that \[\Omega(e_1+e_2)=(-1)^{k-1}\cdot \frac{ k}{2}.\] The most interesting case is $\nu<0$.  When $k=1$  the only nonzero BPS invariants are
\[\Omega(\pm e_1)=\Omega(\pm e_2)=\Omega(\pm (e_1+e_2))=1.\]
The case $k=2$ has the infinite set of nonzero invariants
\[\Omega(\pm (me_1 + ne_2))=1 \text{ if }|m-n|=1, \qquad \Omega(\pm(e_1+e_2))=-2,\]
with all others being zero. 
In general, for $k>2$, these  BPS structures are not very well-understood. However, it is known \cite[Theorem 6.4]{R} that  they are not in general ray-finite. It is also expected  that there exist regions in $\bC^*$ in which the active rays are dense.


\section{Summary of the contents of the paper}

In this section we give a rough summary of  the contents of the rest of the paper.
Precise statements and proofs can be found in later sections.

\subsection{The Riemann-Hilbert problem}
\label{rh_intro}

Given a convergent BPS structure $(\Gamma,Z,\Omega)$ we will 
 consider an associated Riemann-Hilbert problem. It depends on a choice of a point $\xi\in \bT$ which we call the constant term.  We discuss the statement of this problem more carefully in Section \ref{rh}; for now we just give the rough idea.
 
\begin{problem}
\label{an}Fix a point $\xi\in \bT$.
Construct a piecewise holomorphic map
\[\Phi\colon \bC^*\to \bT\]
with the following three properties:
\begin{itemize}
\item[(a)] As $t\in \bC^*$ crosses an active ray $\ell\subset \bC^*$ in the clockwise direction, the function $\Phi(t)$ undergoes a discontinuous  jump described  by the formula
\[\Phi(t)\mapsto \bS(\ell)(\Phi(t)) .\]

\item[(b)] As $t\to 0$ one has $\exp(Z/t)\cdot \Phi(t)\to \xi$.\smallskip

\item[(c)] As $t\to \infty$ the element $\Phi(t)$ has at most polynomial growth.
\end{itemize}

\end{problem}

 The gist of condition (a) is that the function $\Phi$ should be holomorphic in the complement of the active rays, and for each active ray $\ell \subset \bC^*$ the analytic continuations of the two functions on either side  should differ by composition with the corresponding automorphism $\bS(\ell)$. 
 
 To make sense of (b), note that \[\exp(Z/t)\colon \Gamma\to \bC^*\] is an element of the  torus $\bT_+$, and    recall that $\bT=\bT_-$ is a torsor for $\bT_+$. We often write \begin{equation}\exp(Z/t)\cdot \Phi(t)=\Psi(t)\cdot \xi,\end{equation}
which then defines a map $\Psi\colon \bC^*\to \bT_+.$
Clearly the maps $\Phi$ and $\Psi$ are equivalent data; we use whichever is most convenient.

Composing with the (twisted) characters of $\bT_\pm$, we can alternatively encode the solution  in  either of the  two systems of maps
\[\Phi_\gamma
=x_\gamma\circ \Phi\colon \bC^*\to \bC^*, \qquad \Psi_\gamma=y_\gamma\circ \Psi\colon \bC^*\to \bC^*.\]
Condition (c) is then the statement that for each $\gamma\in \Gamma$ there should exist $k>0$ such that
\[|t|^{-k}<|\Phi_\gamma(t)|<|t|^k, \qquad |t|\gg 0.\]

Problem \ref{an} is closely analogous to the Riemann-Hilbert problems which arise in the study of differential equations with irregular singularities. In that case  the Stokes factors $\bS(\ell)$ lie in a finite-dimensional group $\GL_n(\bC)$, whereas in our situation  they are elements of the  infinite-dimensional group of Poisson automorphisms of the torus $\bT$. We will return to this analogy in the sequel to this paper \cite{Br3}.

\subsection{Solution in the uncoupled case}

In the case of a finite, uncoupled,  integral BPS structure, and for certain choices of $\xi\in \bT$, the Riemann-Hilbert problem introduced above has a unique solution, which can be  written explicitly in terms of products of modified gamma functions. The inspiration for this comes from work of Gaiotto \cite[Section 3.1]{G}.

 Consider the multi-valued meromophic function on $\bC^*$ defined by
\begin{equation}\label{who}\Lambda(w)=\frac{e^w \cdot\Gamma(w)}{\sqrt{2\pi}\cdot w^{w-\half}}.\end{equation}
Taking the principal value of $\log$ on $\bC^*\setminus\bR_{<0}$, we consider $\Lambda(w)$ as a single-valued holomorphic function on this domain. The formula \eqref{who} is obtained by exponentiating  the initial terms in the Stirling expansion, and  it follows that $\Lambda(w)\to 1$ as $w\to \infty$ in  any closed subsector of $\bC^*\setminus\bR_{<0}$.

\begin{thm}
\label{one}
Let $(Z,\Gamma,\Omega)$ be a finite,  uncoupled, integral BPS structure. Suppose that $\xi\in \bT$ satisfies $\xi(\gamma)=1$ for all active classes $\gamma\in \Gamma$. Then Problem \ref{an}  has the unique solution
\begin{equation}\Psi_\beta(t)= \prod_{\Im Z(\gamma)/t>0} \Lambda\bigg(\frac{Z(\gamma)}{2\pi i t}\bigg)^{\Omega(\gamma)\<\beta,\gamma\>}
\end{equation}
where the product is  over the finitely many active classes $\gamma\in \Gamma$ with $\Im Z(\gamma)/t>0$.
\end{thm}

Note that there do indeed exist elements $\xi\in \bT$ satisfying the condition of Theorem \ref{one}. Indeed, the uncoupled assumption implies that the active classes $\gamma\in \Gamma$ span a subgroup $\Gamma_a\subset \Gamma$ on which the form $\<-,-\>$ vanishes. Any basis of this primitive subgroup can be combined with the pullback of a basis of the quotient $\Gamma/\Gamma_a$, to give a basis $(\gamma_1,\cdots,\gamma_n)$ for $\Gamma$, and an element $\xi\in \bT$ can be defined by making arbitrary choices $\xi(\gamma_i)\in \{\pm 1\}$. 
The proof of Theorem \ref{one} is a good  exercise in the basic  properties of the gamma function. The details are given in Section \ref{sol}.

\subsection{Variations of BPS structure}
\label{vbpss}

The variation of BPS invariants in Donaldson-Thomas theory under changes in stability parameters is controlled by the Kontsevich-Soibelman wall-crossing formula. This forms the main ingredient in the following  definition of a variation of BPS structures, which is a special case of the  notion of a continuous family of stability structures  from \cite[Section 2.3]{KS1}.
Full details can be found in Appendix \ref{two}; here we just give the rough idea.

\begin{defn}
\label{co}
A variation of BPS structure over a complex manifold $M$ consists of a collection of BPS structures $(\Gamma_p,Z_p,\Omega_p)$ indexed by the points  $p\in M$, such that
\begin{itemize}
\item[(a)] The charge lattices  $\Gamma_p$ form a local system of abelian groups, and the intersection forms $\<-,-\>_p$ are covariantly  constant.\smallskip

\item[(b)] Given a covariantly constant family of elements $\gamma_p\in \Gamma_p$, the central charges $Z_p(\gamma_p)\in \bC$ are holomorphic functions of $p\in M$.\smallskip

\item[(c)] The constant in the support property \eqref{support} can be chosen uniformly on compact subsets.\smallskip

\item[(d)] For each acute sector $\Delta\subset \bC^*$, the anti-clockwise product over active rays in $\Delta$\begin{equation}\label{wcw} \bS_p(\Delta)=\prod_{\ell\subset \Delta} \bS_p(\ell) \in \Aut(\bT_p),\end{equation}
is covariantly constant as $p\in M$ varies, providing the boundary rays of $\Delta$ are never active.
\end{itemize}
\end{defn}

Note that the local system in (a) induces a flat Ehresmann connection on the bundle of tori \[(\bT_+)_p=\Hom_\bZ(\Gamma_p,\bC^*),\]
 over the manifold $M$, and hence also on the associated bundle of twisted tori $\bT_p$.
It will require some work to make rigorous sense of the wall-crossing formula, condition (d). We do this using formal completions in Appendix \ref{two} following  \cite{KS1}. This needs no convergence assumptions and completely describes the behaviour of the BPS invariants as the point $p\in M$ varies: knowing all the  invariants $\Omega(\gamma)$ at some point of $M$ determines them at all other points.

The material of Section \ref{newnew} gives interesting examples of variations of BPS structures. For a fixed $k\geq 1$ the corresponding BPS structure $(\Gamma,Z,\Omega)$ is determined by the pair \[(Z(e_1), Z(e_2))\in \mathfrak{h}^2,\] where $\mathfrak{h}$ denotes the upper half-plane.  As this point varies we obtain a variation of BPS structures. In particular, the BPS invariants for $\nu<0$ are completely determined by the trivial case $\nu>0$ and the wall-crossing formula \eqref{wcw}.

\subsection{Tau functions}

Let us consider a variation of BPS structures $(\Gamma_p,Z_p,\Omega_p)$ over a complex manifold $M$. We call such a variation framed if the local system  $(\Gamma_p)_{p\in M}$ is trivial, so that we can  identify all  the lattices $\Gamma_p$  with a fixed lattice $\Gamma$. We can always reduce to this case by passing to a cover of $M$, or by restricting to a neighbourhood of a given point $p\in M$.

Associated to a framed variation $(\Gamma,Z_p,\Omega_p)$ there is a holomorphic map
\[\pi\colon M\to \Hom_\bZ(\Gamma,\bC)\isom \bC^n, \quad p\mapsto Z_p\]
which we call the period map.
We say that the variation is miniversal if the period map is  a local isomorphism. In that case, if we choose a basis $(\gamma_1,\cdots, \gamma_n)\subset \Gamma$, the functions $z_i=Z(\gamma_i)$ form a system of local co-ordinates in a neighbourhood of any given point of $M$.

Consider a framed, miniversal variation $(\Gamma_p,Z_p,\Omega_p)$ over a manifold $M$, and choose a basis $(\gamma_1,\cdots, \gamma_n)\subset \Gamma$ as above. For  each point $p\in M$ we can consider the Riemann-Hilbert problem associated to the BPS structure $(\Gamma,Z_p,\Omega_p)$. The wall-crossing formula,  Definition \ref{co} (d), makes it reasonable to ask for a family of solutions to these problems which is a piecewise holomorphic function of $p\in M$.
Such a family of solutions is given by a piecewise holomorphic map
\[\Psi\colon M\times \bC^*\to \bT_+,\]
 which we view as a  function of the co-ordinates  $(z_1,\cdots, z_n)\in \bC^n$ and the parameter $t\in \bC^*$.  We define a  $\tau$-function for the given family of solutions to be a piecewise holomorphic function \[\tau\colon M\times \bC^* \to \bC^*,\]
 which is invariant under simultaneous  rescaling of all co-ordinates $z_i$ and the parameter $t$, and which satisfies  the equations
 \begin{equation}
\label{hungry}\frac{1}{2\pi i}\cdot \frac{\partial \log \Psi_{\gamma_j}}{\partial t}=\sum_i \epsilon_{ij}  \frac{\partial \log \tau}{\partial z_i}, \end{equation}
where $\epsilon_{ij}=\<\gamma_i,\gamma_j\> $.
When the form $\<-,-\>$ is non-degenerate these conditions uniquely determine $\tau$ up to multiplication by a constant scalar factor.

It is not clear at present why a $\tau$-function should exist in general, and the above definition should be thought of as being somewhat experimental. Nonetheless, in the uncoupled case we will see that $\tau$-functions do exist, and are closely related to various partition functions arising in quantum field theory. We hope to return to the general case in future publications.
 
 \subsection{Tau function in the uncoupled case}
Suppose given a miniversal variation of finite, uncoupled, integral BPS structures over a complex manifold $M$. We  will show that the family of  solutions given by Theorem \ref{one} has a corresponding $\tau$-function. To describe this function  we first introduce the expression
\begin{equation}\Upsilon(w)=\frac{e^{-\zeta'(-1)}\, e^{\frac{3}{4}w^2} \,G(w+1)}{(2\pi)^{\frac{w}{2}}\,w^{\frac{w^2}{2}}},\end{equation}
where $G(x)$ is the Barnes $G$-function, and $\zeta(s)$ the Riemann zeta function. 

\begin{thm}
\label{doss}
Let $(\Gamma,Z_p,\Omega_p)$ be a framed, miniversal variation of finite, uncoupled, integral BPS structures over a complex manifold $M$. Then the function  
\begin{equation}\tau(Z,t)=\prod_{\Im Z(\gamma)/t>0} \Upsilon\bigg(\frac{Z(\gamma)}{2\pi i t}\bigg)^{\Omega(\gamma)}\end{equation}
is a $\tau$-function for the family of solutions given by Theorem \ref{one}.\end{thm}

The known asymptotics  of the $G$-function imply  that $\tau(Z,t)$  has  an  asymptotic expansion involving the Bernoulli numbers
\begin{equation}\log \tau(Z,t)\sim \frac{1}{24} \sum_{\gamma\in \Gamma} \Omega(\gamma)\log \bigg(\frac{2\pi i t}{Z(\gamma)}\bigg) +\sum_{g\geq 2} \sum_{\gamma\in \Gamma}   \frac{\Omega(\gamma)\cdot B_{2g}}{4g\, (2g-2)} \bigg(\frac{2\pi i t}{Z(\gamma)}\bigg)^{2g-2}\end{equation}
valid as $t\to 0$ in any  half-plane whose boundary rays are not active.

\subsection{Two classes of examples}

There are two important classes of examples of BPS structures where we can say something about solutions to the Riemann-Hilbert problem. Both are also of interest in theoretical physics. They are treated a little more thoroughly in Sections \ref{gw} and \ref{in}, but  there are many unanswered questions which we leave for future research.

\subsection* {\it (i) Topological strings.} Let $X$ be a compact Calabi-Yau threefold. There is a variation of BPS structures over the complexified K{\" a}hler cone
\[\{\omega_\bC=B+i\omega\in H^2(X,\bC): \omega\text{ ample}\},\]
arising  mathematically from generalised Donaldson-Thomas theory applied to coherent sheaves on $X$ supported in dimension $\leq 1$. The  BPS invariants  are expected to coincide with the genus 0 Gopakumar-Vafa invariants $\GV(0,\beta)$ (see \cite[Conjecture 6.20]{JS}). Assuming this, we argue  that the asymptotic expansion of the resulting $\tau$-function should be related to the Gromov-Witten partition function of $X$. More precisely, it should reproduce the $g\geq 2$ terms in those parts of the  partition function arising from constant maps and genus 0 degenerate contributions:
\[ \log \tau(\omega_\bC,t)\sim \chi(X)
\sum_{g\geq 2}   \frac{ B_{2g} \,B_{2g-2} \, (2\pi t)^{2g-2} }{4g\, (2g-2) \,(2g-2)!}+
\sum_{\stackrel{g\geq 2}{\beta\in H_2(X,\bZ)}}\GV(0,\beta) \, \frac{B_{2g} \, (2\pi t)^{2g-2} }{2g\, (2g-2)!  }\Li_{3-2g}(e^{2\pi i \omega_\bC\cdot \beta}).\]
We give a complete proof of this result in the case of the resolved conifold in \cite{tocome}: this involves writing down a non-perturbative version of the above expression and checking that it gives rise via \eqref{hungry} to a solution to the Riemann-Hilbert problem. 
\subsection* {\it (ii) Theories of class $S$} Our second  example relates to the class of $N=2$, $d=4$  gauge theories known as theories of class $S$. We consider only the case of  gauge group $\SU(2)$. To specify the theory we need to fix a genus $g\geq 0$ and a collection of $d\geq 1$ integers  \[m=(m_1,\cdots,m_d),  \quad m_i\geq 2.\]
 Mathematically, we can then proceed by introducing a complex orbifold  $\Quad(g,m)$ parameterizing pairs $(S,q)$ consisting of a Riemann surface $S$ of genus $g$, and a meromorphic quadratic differential $q$ on $S$ which has simple zeroes, and poles of the given multiplicities $m_i$. It is proved in \cite{BS} that this  space also  arises as a (discrete quotient of) the space of stability conditions on a  triangulated category $\cD(g,m)$ having the three-dimensional Calabi-Yau property. Applying generalized Donaldson-Thomas theory then leads to a variation of BPS structures over $\Quad(g,m)$, whose central charge is given by the periods of the differential $q$, and whose BPS invariants are counts of finite-length trajectories.
 Work of Iwaki and Nakanishi \cite{IN} shows  that the Riemann-Hilbert problem corresponding to a pair $(S,q)$ is closely related to exact WKB analysis of  the corresponding Schr{\"o}dinger equation
 \[\hbar^2 \,\frac{d^2}{dz^2} y(z;\hbar)= q(z) y(z,\hbar),\]
 where $z$ is some local co-ordinate in a fixed projective structure on $S$, and $\hbar$ should be identified with the variable $t$ in Problem \ref{an}. This story is the conformal limit of that described by Gaiotto, Moore and Neitzke in the paper \cite{GMN2}.


\section{The Riemann-Hilbert problem}
\label{rh}

In this section we  discuss the Riemann-Hilbert problem defined by a convergent BPS structure.  It is closely related to the Riemann-Hilbert problem considered by Gaiotto, Moore and Neitzke \cite{GMN1}, and has also been studied by Stoppa and his collaborators \cite{BSt,FGS}.

\subsection{Analytic BPS automorphisms}
\label{ref}
We begin by summarising  some basic analytic facts which are proved in Appendix \ref{converge}. 
Let $(\Gamma, Z, \Omega)$ be a  BPS structure, let  $\bT$ be the associated twisted torus, and fix an acute sector $\Delta\subset \bC^*$. For each real number $R>0$ define the open subset $U_\Delta(R)\subset \bT$  to be the interior of the subset
\[\big\{g\in \bT: Z(\gamma)\in \Delta \text{ and } \Omega(\gamma)\neq 0\implies |g(\gamma)|<\exp(-R\|\gamma\|)\big\}\subset \bT.\]
It follows from the support property that the subset $U_\Delta(R)\subset \bT$ is non-empty (see Lemma \ref{dyl}).
Recall the definition of the height of a ray $\ell\subset \bC^*$ from Section \ref{raydiagram}. 
 
\begin{prop}
\label{propy}
Let $(\Gamma, Z, \Omega)$ be a convergent BPS structure, and $\Delta\subset \bC^*$ a convex sector.
For sufficiently large $R>0$ the following statements hold:
\begin{itemize}
\item[(i)] For each ray $\ell\subset \Delta$, the series \eqref{hamiltonian} defining $\DT(\ell)$ is absolutely convergent on $U_\Delta(R)$, and hence defines a holomorphic function \[\DT(\ell)\colon U_\Delta(R)\to \bC.\]
\item[(ii)] The time 1 Hamiltonian flow of the function $\DT(\ell)$  with respect to the Poisson structure $\{-,-\}$ on $\bT$ defines a holomorphic  embedding
\[\bS(\ell)\colon U_\Delta(R)\to \bT.\]
\item[(iii)] For each $H>0$, the composition in anti-clockwise order
\[\bS_{<H}(\Delta) = \bS_{\ell_1} \circ \bS_{\ell_2} \circ \cdots \circ \bS_{\ell_k},\]
 corresponding to the finitely many rays $\ell_i\subset \Delta$ of height $< H$ exists, and  the pointwise limit \[\bS(\Delta)=\lim_{H\to \infty} \bS_{<H}(\Delta) \colon U_\Delta(R)\to \bT\]
is a well-defined  holomorphic embedding. 
\end{itemize}
\end{prop}

\begin{proof}
See Appendix \ref{converge}, Proposition \ref{prop}.
\end{proof}

We think of the maps $\bS(\ell)$  of Proposition \ref{propy} as  being partially-defined automorphisms of the twisted torus $\bT$. We will usually restrict attention to BPS structures which, in the terminology of Section \ref{defns}, are ray-finite, generic and integral. The maps  $\bS(\ell)$ can then be computed using the following result.

\begin{prop}
\label{barneyy}
 Suppose that $(\Gamma, Z, \Omega)$ is ray-finite, generic and integral.  Then for any ray $\ell\subset \bC^*$ the   embedding  $\bS(\ell)$ of Proposition \ref{propy} extends to a birational automorphism of $\bT$, whose action on twisted characters is given by 
\begin{equation}\bS(\ell)^*(x_\beta)=x_\beta\cdot \prod_{Z(\gamma)\in \ell}(1-x_\gamma)^{\,\Omega(\gamma)\<\gamma,\beta\>}.\end{equation}
\end{prop}

\begin{proof}
See Appendix  \ref{converge}, Proposition \ref{barney}.
\end{proof}

Note that if the BPS structure $(\Gamma,Z,\Omega)$ satisfies the stronger condition of being finite, then there are only finitely many active rays, so for any acute sector $\Delta\subset \bC^*$ the map $\bS(\Delta)$ of Proposition \ref{propy} also extends to a birational automorphism of $\bT$.

\subsection{Statement of the problem}
\label{real}

Let $(\Gamma, Z, \Omega)$ be a convergent BPS structure, and denote by $\bT$ the associated twisted torus. Given a ray $\rr\subset \bC^*$ we consider the corresponding half-plane
\[\bH_\rr=\{t\in \bC^*:t=z\cdot v \text{ with } z\in \rr\text{ and }\Re(v)>0\}\subset \bC^*\]
centered on it.
We shall be dealing with functions of the form \[\Phi_\rr\colon \bH_\rr\to  \bT.\] Composing with the twisted characters of $\bT$ we can equivalently consider functions
\[\Phi_{\rr,\gamma}\colon \bH_\rr\to \bC^*,\qquad \Phi_{\rr,\gamma}(t)=x_\gamma(\Phi_\rr(t)).\]

The Riemann-Hilbert problem associated to the BPS structure $(\Gamma, Z, \Omega)$ depends on a choice of element $\xi\in \bT$ which we call the constant term. It reads as follows:

\begin{problem}
\label{dtsect}
Fix an element $\xi\in \bT$. For each non-active ray $\rr\subset \bC^*$, we  seek a holomorphic function
$\Phi_\rr\colon \bH_\rr\to \bT$
such that the following three conditions are satisfied:
\begin{itemize}

\item[(RH1)] {\it Jumping.} Suppose that two  non-active rays $\rr_-,\rr_+\subset \bC^*$ form the  boundary rays of a convex sector $\Delta\subset \bC^*$, taken in clockwise order. Then
\[\Phi_{\rr_+}(t)= \bS(\Delta)( \Phi_{\rr_-}(t)), \]
for all $t\in \bH_{\rr_-}\cap \bH_{\rr_+}$ with  $ 0<|t|\ll 1$.
\smallskip

\item[(RH2)] {\it Finite limit at $0$}.
For each non-active ray $\rr\subset \bC^*$, and each class $\gamma\in \Gamma$, we have
\[\exp(Z(\gamma)/t)\cdot \Phi_{\rr,\gamma}(t) \to \xi(\gamma)\]
as $t\to 0$ in the half-plane $\bH_\rr$.\smallskip

\item[(RH3)] {\it Polynomial growth at $\infty$}. For any class $\gamma\in \Gamma$, and any non-active ray $\rr\subset \bC^*$, there exists $k>0$ such that
\[|t|^{-k} < |\Phi_{\rr,\gamma}
(t)|<|t|^k,\]
for  $t\in \bH_\rr$ with  $|t|\gg 0$.
\end{itemize}
\end{problem}

To make sense of the condition (RH1) note that by Proposition \ref{barneyy} we can find  $R>0$ such that the partially-defined automorphism $\bS(\Delta)$ is well-defined on the open subset $U_{\Delta}(R)\subset \bT$. Observe that if an active class $\gamma\in \Gamma$ satisfies $Z(\gamma)\in \Delta$, and we take $t\in \bH_{\rr_1}\cap \bH_{\rr_2}$, then the quantity $Z(\gamma)/t$ has strictly positive real part.  Using the support property it follows (see the proof of Lemma \ref{dyl}) that for  $ 0<|t|\ll 1$
\[\exp(-Z/t)\cdot \xi\in U_\Delta(R).\]
Condition (RH2) then implies that  $\Phi_{\rr_i}(t)\in U_\Delta(R)$ whenever  $t\in \bH_{\rr_1}\cap \bH_{\rr_2}$ with  $ 0<|t|\ll 1$, so that 
 the  relation (RH1) is indeed well-defined.

\begin{remark}
\label{wet}
When the BPS structure $(\Gamma,Z,\Omega)$ is finite, integral and generic, we can rewrite the condition (RH1) using Proposition \ref{barneyy}. Given an active ray $\ell$, consider a small anticlockwise perturbation $\rr_-$, and  a small clockwise perturbation $\rr_+$, both rays being  non-active.  Then  (RH1) is the condition that
\begin{equation}
\label{garage}\Phi_{\rr_+,\beta}(t)=\Phi_{\rr_-,\beta}(t)\cdot \prod_{Z(\gamma)\in \ell}   (1-\Phi_{\rr_{\pm,\gamma}}(t))^{\Omega(\gamma)\<\gamma,\beta\>}.\end{equation}
Note that the generic assumption ensures there is no need to distinguish the functions $\Phi_{\rr_\pm}(t)$ inside the product, since for classes $\gamma\in \Gamma$ satisfying $Z(\gamma)\in \ell$ they are equal.
\end{remark}

It will be useful to consider the maps $\Psi_\rr\colon \bH_\rr\to \bT_+$ defined by
\[\exp(Z/t)\cdot \Phi_\rr(t)=\Psi_\rr(t)\cdot \xi.\]
 Composing with the characters of $\bT_+$ we can also encode the solution in the system of maps
\[\Psi_{\rr,\gamma}(t)=y_\gamma(\Psi_\rr(t))=\exp(Z(\gamma)/t)\cdot \Phi_{\rr,\gamma}(t) \cdot \xi(\gamma)^{-1}.\]
Of course the maps $\Phi_\rr$ and $\Psi_\rr$ are equivalent data: we use whichever is most convenient.

\begin{remark}
\label{whoopsy}
It follows from the condition (RH1) that if two non-active rays $\rr_1,\rr_2$ bound a convex sector containing no active rays, then the two functions $\Phi_{\rr_i}\colon \bH_{\rr_i}\to \bT$ required in Problem \ref{dtsect} glue together to give a holomorphic  function on $\bH_{\rr_1}\cup\bH_{\rr_2}$. It follows that if a non-active ray $\rr\subset \bC^*$ is  not a limit of active rays, then  the corresponding function  $\Phi_\rr$  extends analytically to a neighbourhood of the closure of  $\bH_\rr\subset \bC^*$.
\end{remark}

\subsection{Remarks on the formulation}

The Riemann-Hilbert problem of the last subsection is the main subject of this paper.  Unfortunately we have no general results giving existence or uniqueness of its solutions. Moreover one could easily imagine various small perturbations of the statement of Problem \ref{dtsect}, and it will require further work to decide for sure exactly what the correct conditions should be. We make a few remarks on this here.

\begin{remarks}
\label{oh}
\begin{itemize}
\item[(i)] From a heuristic point-of-view it is useful to consider a Riemann-Hilbert problem involving maps from $\bC^*$ into the group $G$ of Poisson automorphisms of the torus $\bT$. The  above formulation is obtained by evaluating a $t$-dependent automorphism of $\bT$ at the chosen point $\xi\in \bT$. If we  replace the infinite-dimensional group $G$ with the finite-dimensional group $\GL_n(\bC)$ the analogous Riemann-Hilbert problems are familiar in the theory of linear differential equations with irregular singularities, and play an important r{o}le in the theory of Frobenius manifolds  \cite[Lecture 4]{D}. This connection between stability conditions and Stokes phenomena goes back to \cite{BTL}, and will be revisited in \cite{Br3}. In fact, with our conventions, the Riemann-Hilbert problem is most naturally viewed as taking values in the opposite group $G^{\operatorname{op}}$, which we can view heuristically  as the group of automorphisms of the Poisson algebra $\bC[\bT]$. This is why the compositions in \eqref{wcw} are in the anti-clockwise order, and why the Stokes data acts on the left in (RH1).\smallskip

\item[(ii)] 
In Section \ref{rh_intro} we gave a simplified formulation of the Riemann-Hilbert problem  which  considers a single function $\Phi\colon \bC^*\to \bT$ with prescribed discontinuities along active rays.  This  becomes a little tricky to make sense of when the active rays are dense in regions of $\bC^*$, so we prefer the formulation given in Problem \ref{dtsect}, which is modelled on the standard approach in the finite-dimensional case.
We can obtain a solution to Problem \ref{an} from a solution to Problem \ref{dtsect} by defining $\Phi(t)=\Phi_{\bR_{>0}\cdot t} (t)$; Remark \ref{whoopsy} shows that this defines  a holomorphic function  away from the closure of the union of the active rays. Note that Problem \ref{dtsect}  imposes strictly stronger conditions on the resulting function $\Phi(t)$, because  the conditions (RH2) and (RH3) are assumed to hold in half-planes. 
\smallskip

\item[(iii)]We can weaken the conditions in Problem \ref{dtsect} in various ways, and until we have studied more examples in detail it is not possible to be sure exactly what is the correct formulation. For example, we could allow the functions $\Phi_{\rr,\gamma}(t)$ to have poles on the half-plane $\bH_\rr$, or we could replace $\bH_\rr$ by a smaller convex sector of some fixed angle.  We can also consider a variant of Problem \ref{dtsect} where we only assume that the map $\Phi_\rr$ is defined and holomorphic on the intersection of $\bH_\rr$ with some punctured disc $\{t\in \bC^*:|t|<r\}$, and drop condition (RH3) altogether. We shall refer to this last version as the  weak Riemann-Hilbert problem associated to the BPS structure. 
\end{itemize}
\end{remarks}

\subsection{Symmetries of the problem}
\label{sym}
There are a couple of obvious symmetries of the Riemann-Hilbert problem which deserve further comment.
For the first, note that the twisted torus $\bT$ has a canonical involution $\sigma\colon \bT\to \bT$  which acts on twisted characters by $\sigma^*(x_\gamma)=x_{-\gamma}$.  The fixed point set  is the finite subset
\begin{equation}\label{quadref}\bT^\sigma=\{g\colon \Gamma \to \{\pm 1\}: g(\gamma_1+\gamma_2)=(-1)^{\<\gamma_1,\gamma_2\>} g(\gamma_1)\cdot g(\gamma_2)\}\subset \bT,\end{equation}
whose elements are 
quadratic refinements of the form $\<-,-\>$. 

The symmetry property $\Omega(-\gamma)=\Omega(\gamma)$ of the BPS invariants implies that 
\[\bS(-\ell)\circ \sigma =\sigma\circ \bS(\ell).\]
 It follows that  given a collection of functions  $\Phi^{\xi}_\rr(t)$ solving the Riemann-Hilbert problem for the constant term $\xi$, we can generate another solution, this time for the constant term $\sigma(\xi)$, by defining \[{\Phi}^{\sigma(\xi)}_{\rr}(t)=\sigma\circ \Phi^\xi_{-\rr}(-t).\]
 In particular,  if we had uniqueness of solutions, we could  conclude that  whenever $\xi\in \bT^\sigma$ is a quadratic refinement of the form $\<-,-\>$, any solution to the  Riemann-Hilbert problem satisfies
 \[\Phi^\xi_{\rr,\gamma}(t)= \Phi^\xi_{-\rr,-\gamma}(-t).\]
 
 For the second symmetry of the Riemann-Hilbert problem, note that given a BPS structure $(\Gamma,Z,\Omega)$, we can obtain a new BPS structure $(\Gamma,\lambda Z,\Omega)$ by simply multiplying the central charges $Z(\gamma)\in \bC$ by a fixed scalar $\lambda\in \bC^*$. It is then clear that if $\Phi_\rr(t)$ is a system of solutions to the Riemann-Hilbert problem  for the BPS structure $(\Gamma,Z,\Omega)$, then the functions $\Phi_{\lambda\rr}(\lambda t)$ will be a system of solutions to the problem for $(\Gamma,\lambda Z,\Omega)$ with the same constant term. Again, if we had uniqueness of solutions, we could conclude that all solutions to Problem \ref{dtsect} are invariant under simultaneous rescaling of $Z$ and $t$.

\subsection{Null vectors and uniqueness}

Let  $(Z,\Gamma,\Omega)$ be a  BPS structure, and denote by $\bT$ the corresponding twisted torus.
\begin{defn}
An element $\gamma\in \Gamma$ will be called null if it satisfies $\<\alpha,\gamma\>=0$ for all active classes $\alpha\in \Gamma$. A twisted character $x_\gamma\colon \bT\to \bC^*$ corresonding to a null element $\gamma\in \Gamma$  will be called a  coefficient.
\end{defn}

Note that the definition of the wall-crossing automorphisms $\bS(\ell)$ shows that they fix all coefficients:
$\bS(\ell)^*(x_\gamma)=x_\gamma$. This leads to the following partial uniqueness result.

\begin{lemma}
\label{null}
Let $(Z,\Gamma,\Omega)$ be a convergent BPS structure and $\gamma\in \Gamma$ a null element. Then for any solution to the Riemann-Hilbert problem, and any non-active ray $\rr\subset \bC^*$, one has 
$\Psi_{\rr, \gamma}(t)=1$
for all $t\in \bH_\rr$.
\end{lemma}

\begin{proof}
Since coefficients are unchanged by wall-crossing, condition (RH1) shows that the functions $\Psi_{\rr,\gamma}(t)$ for different rays $\rr\subset \bC^*$ piece together to give a single holomorphic function $\Psi_\gamma\colon \bC^*\to \bC^*$. Since we can cover $\bC^*$ by  the half-planes  $\bH_{\rr_i}$ defined by finitely many non-active rays $\rr_i\subset \bC^*$, condition (RH2) shows that this function has a removable singularity at $0\in \bC$ with value $\Psi_\gamma(0)=1$, and condition (RH3) shows that it has at worst polynomial growth at $\infty$. It follows that $\Psi_\gamma$ extends to a meromorphic function $\bC\bP^1\to \bC\bP^1$ which has neither zeroes nor poles on $\bC$. This implies that $\Psi_\gamma(t)$ is constant, which proves the result.
\end{proof}

Recall the definition of an uncoupled BPS structure from Section \ref{defns}. 

\begin{lemma}
\label{brownpaperbag}
Let $(Z,\Gamma,\Omega)$ be a finite, uncoupled, integral  BPS structure. Then the associated Riemann-Hilbert problem has at most one solution.
\end{lemma}

\begin{proof}
Note that uncoupled BPS structures are in particular generic, so Remark \ref{wet} applies.
Fix an arbitrary class $\beta\in \Gamma$. The uncoupled condition implies that all classes $\gamma\in \Gamma$ appearing in the product on the right-hand side of \eqref{garage} are null. Lemma \ref{null} shows that for these classes  the corresponding functions
\[\Phi_{\rr,\gamma}(t)=\exp(-Z(\gamma)/t)\cdot \xi(\gamma)\]
are uniquely determined. It follows that if we have two systems of solutions \[\Phi^{(i)}_{\rr,\beta}(t)\colon \bH_\rr\to \bC^*, \quad i=1,2,\] to the Riemann-Hilbert problem, then the ratios\[q_{\rr,\beta}(t)=\Phi^{(2)}_{\rr,\beta}(t)\cdot (\Phi^{(1)}_{\rr,\beta}(t))^{-1} \colon \bH_\rr\to \bC^*\]
piece together to give a single holomorphic function $q_\beta\colon \bC^*\to \bC^*$. Arguing exactly as in Lemma \ref{null},  we can use conditions (RH2) and (RH3) to conclude that $q_\beta(t)=1$ for all $t\in \bC^*$, which proves the claim.
\end{proof}

\begin{remark}
Let $(\Gamma,Z,\Omega)$ be a convergent BPS structure, and consider the doubled BPS structure $(\Gamma\oplus \Gamma^\dual,Z,\Omega)$ of  Section \ref{double}.  For each class $\gamma\in \Gamma$ the  corresponding element \[\gamma_D=(\gamma, \<-,\gamma\>)\in \Gamma\oplus \Gamma^\dual\]
is  null, because, by the formula \eqref{job},   it is orthogonal to any class of the form $(\alpha,0)\in \Gamma\oplus \Gamma^\dual$. Lemma \ref{null} therefore implies that any solution to the Riemann-Hilbert problem for the double satisfies $\Psi_{\gamma_D,\rr}(t)=1$. This implies that
\[\Psi_{\rr, (0,\<\gamma,-\>)}(t)=\Psi_{\rr,(\gamma,0)}(t),\]
 for all  non-active rays $\rr\subset \bC^*$, and all $t\in \bH_r$. 
 
  In this way one sees that  if a BPS structure has a non-degenerate form $\<-,-\>$, then solving the Riemann-Hilbert problem for the doubled BPS structure is precisely equivalent to solving the problem for the original BPS structure. However, when the form $\<-,-\>$ is degenerate,  a solution to the Riemann-Hilbert problem for the doubled BPS structure contains strictly more information than a solution to the  original problem. We will see an example of this in Section \ref{a1} below. \end{remark}

\subsection{Tau functions} 

Consider a framed, miniversal variation of BPS structures $(\Gamma,Z_p,\Omega_p)$  over a complex manifold $M$.
For the relevant definitions the reader can either consult the summary in Section \ref{vbpss}, or the full treatment in Appendix \ref{two}. There is a holomorphic map \[\pi \colon M\to \Hom_\bZ(\Gamma,\bC), \quad p\mapsto Z_p,\]
which we call the period map. The miniversal assumption ensures that the derivative
\begin{equation}
\label{friday}(D\pi)_p \colon T_p M\to \Hom_{\bZ}(\Gamma,\bC)\end{equation}
at each point $p\in M$ is an isomorphism. We  can  therefore identify the tangent space $T_p M$    with the  vector space $\Hom_\bZ(\Gamma,\bC)$.   The  form $\<-,-\>$  then induces a Poisson structure on $M$.

More explicitly, we can choose a basis $(\gamma_1,\cdots,\gamma_n)\subset \Gamma$ and use the functions \[z_i(p)=Z_p(\gamma_i)\colon M\to \bC\]
as  co-ordinates on $M$. The dual of the map \eqref{friday}  identifies  $\gamma_i\in \Gamma$ with the element $dz_i\in T_p^*M$, and the Poisson structure   has the Darboux form
\[\{z_i,z_j\}=\epsilon_{ij}, \qquad \epsilon_{ij}=\<\gamma_i,\gamma_j\>.\]

Let us  consider a family of solutions \[\Psi_\rr=\Psi_\rr(p,t)\colon M\times \bC^*\to\bT_+\]
 to the  Riemann-Hilbert problems associated to the BPS structures $(\Gamma,Z_p,\Omega_p)$. For each ray $\rr\subset \bC^*$ we assume that $\Psi_\rr(p,t)$ is a piecewise holomorphic function, with discontinuities at points $p\in M$ where the ray $\rr\subset \bC^*$ is active in the BPS structure $(\Gamma,Z_p,\Omega_p)$. Differentiating with respect to $t$ and translating to the identity $1\in \bT_+$ we get a map
\[\Psi_\rr^{-1}\cdot \frac{d\Psi_\rr}{dt}\colon M\times \bC^*\to \Hom_\bZ(\Gamma,\bC).\]
Composing with  the inverse of \eqref{friday}, this can be viewed as a vector field $V(t)$ on $M$ depending on $t\in \bC^*$. A $\tau$-function for the family of solutions $\Psi_\rr(p,t)$  is a piecewise holomorphic function \[\tau_\rr=\tau_\rr(p,t)\colon M\to \bC^*,\]
 such that $V(t)$ is the Hamiltonian vector field of the function $(2\pi i)\cdot \log \tau_\rr$. In terms of the co-ordinates $z_i=Z(\gamma_i)$ described above,  the condition is that
 \begin{equation}
 \label{nick2}\frac{1}{2\pi i}\cdot\frac{\partial \log \Psi_{\rr,\gamma_j}}{\partial t}=\sum_i \epsilon_{ij}  \frac{\partial \log \tau_\rr}{\partial z_i}.\end{equation}
By the second symmetry property of Section \ref{sym} it is natural to also impose the condition that  $\tau$ is invariant under simultaneous  rescaling of $Z$ and $t$. In the case when the form $\<-,-\>$ is non-degenerate this is enough to determine the  $\tau$-function uniquely up to multiplication by an element of $\bC^*$.


\section{Explicit solutions in the finite, uncoupled case}
\label{sol}

In this section we show how to solve the Riemann-Hilbert problem associated to a finite, uncoupled, integral BPS structure, and compute the $\tau$-function associated to a variation of such structures. The situation considered here corresponds to the case of `mutually local corrections' in \cite{GMN1}. The inspiration for our solution  comes from a calculation of Gaiotto \cite[Section 3]{G}.

\subsection{Doubled A$_1$ example}
\label{a1}

The following BPS structure arises from  the A$_1$ quiver, which consists of a single vertex and no arrows. It depends on a parameter $z\in \bC^*$.

\begin{itemize}
\item[(i)] The lattice $\Gamma=\bZ\cdot \gamma$ has rank one, and  thus $\<-,-\>=0$;
\smallskip
\item[(ii)] The central charge $Z\colon \Gamma\to \bC$ is determined by $Z(\gamma)=z\in \bC^*$;
\smallskip
\item[(iii)] The only non-vanishing BPS invariants are $\Omega(\pm\gamma)=1$. 
\end{itemize}

The Riemann-Hilbert problem associated to this BPS structure is trivial since all elements of $\Gamma$ are null. Let us instead consider the  double  of this BPS structure \[(\Gamma_D,Z,\Omega), \quad \Gamma_D=\Gamma\oplus\Gamma^\dual.\]
Let $\gamma^\dual\in \Gamma^\dual$ be the unique generator satisfying $\gamma^\dual(\gamma)=1$. Note that the definition  \eqref{job} shows that the skew-symmetric form $\<-,-\>$ on $\Gamma_D$  satisfies $\<\gamma^\dual,\gamma\>=1$.  

To define the Riemann-Hilbert problem for the doubled BPS structure we must first choose a constant term $\xi_D\in \bT_D$, where $\bT_D$ is the twisted torus corresponding to the lattice $\Gamma_D$.
For simplicity we take $\xi_D(\gamma^\vee)=1$, and write  $\xi=\xi_D(\gamma)\in \bC^*$.

The only active rays  are $\ell_\pm=\pm \bR_{>0}\cdot z$, and 
 Lemma \ref{null} shows that
\[\Phi_{\rr,\gamma}(t)=\exp(-z/t)\cdot \xi\]
for any non-active ray $\rr\subset \bC^*$.
The non-trivial part of the Riemann-Hilbert  problem for the doubled BPS structure consists of the functions
\[ \Psi_\rr(t)=\Psi_{\rr, \gamma^\dual}(t)\colon \bH_\rr\to \bC^*.\]

It follows from Remark \ref{whoopsy}  that the functions $\Psi_{\rr}(t)$ for non-active rays $\rr\subset \bC^*$ lying in the same component of $\bC\setminus \bR\cdot z$ are analytic continuations of each other. Thus we obtain just two holomorphic functions
\[\Psi_\pm(t)\colon \bC^*\setminus  i\ell_{\pm}   \to  \bC^*,\]
 corresponding to non-active rays lying in the half-planes $\pm \Im(z/t)>0$. 
 Using Remark \ref{wet}, the Riemann-Hilbert problem for the doubled BPS structure can therefore be restated as follows.

\begin{problem}
\label{eps}
 Find holomorphic functions $\Psi_\pm\colon \bC^*\setminus i\ell_\pm \to \bC^*$
  such that:
\begin{itemize}
\item[(i)]
There are relations
\begin{equation}
\label{rabbit}\Psi_+(t)=\begin{cases} \Psi_-(t) \cdot (1-\xi^{+1}\cdot e^{-z/t})^{-1}&\mbox{ if }t\in \bH_{\ell_+},\\ \Psi_-(t) \cdot (1-\xi^{-1}\cdot e^{+z/t})^{-1}&\mbox{ if }t\in \bH_{\ell_-}.\end{cases}\end{equation}

\item[(ii)] As $t\to 0$ in $\bC^*\setminus i\ell_{\pm}$ we have $\Psi_{\pm}(t)\to 1$.

\item[(iii)] There exists $k>0$ such that
\[|t|^{-k} < |\Psi_{\pm}(t)|<|t|^k\]
as $t\to \infty$ in $\bC^*\setminus i\ell_{\pm}$.
\end{itemize}
\end{problem}

To understand condition (i) note that if $t\in \bH_{\ell_+}$ then $t$ lies in the domains of definition $\bH_\rr$ of the functions $\Psi_\rr$ corresponding to sufficiently small deformations of the ray $\ell_+$. Thus \eqref{garage} applies to the ray $\ell_+$ and we obtain the first of  the relations \eqref{rabbit}. The second follows similarly from \eqref{garage} applied to the opposite ray $\ell_-$.

\subsection{Solution in the doubled A$_1$ case}

The formula
\begin{equation}
\label{cheedale}\Lambda(w)=\frac{e^w \cdot\Gamma(w)}{\sqrt{2\pi}\cdot w^{w-\half}}\end{equation}
defines a  meromorphic  function of $w\in \bC^*$, which is multi-valued due to the   factor
\[w^{w-\half}=\exp\big({(w-\half)\log w}\big).\]
Since $\Gamma(w)$ is meromorphic on $\bC$ with poles only at the non-positive integers, it follows that
$\Lambda(w)$ is holomorphic on $\bC^*\setminus\bR_{<0}$. We specify it uniquely by taking the principal branch of $\log$.

The Stirling expansion \cite[Section 12.33]{WW} gives an asymptotic expansion
\begin{equation}\label{stirling}\log \Lambda(w)
\sim  \sum_{g=1}^{\infty} \frac{B_{2g}}{2g(2g-1)}w^{1-2g},
\end{equation}
where $B_{2g}$ denotes the $(2g)$th Bernoulli number. This expansion is valid as $w\to \infty$ in the complement of a closed sector containing the ray  $\bR_{<0}$. It implies in particular that $\Lambda(w)\to 1$.

\begin{prop}
\label{twotier}
 When $\xi=1$, Problem \ref{eps} has a unique solution, namely 
\[\Psi_+(t)=\Lambda\Big(\frac{z}{2\pi i t}\Big), \qquad \Psi_-(t)=\Lambda\Big(\frac{-z}{2\pi i t}\Big)^{-1}.\]
\end{prop}

\begin{proof}
Note that the function $\Psi_+(t)$ is indeed holomorphic on the domain $\bC^*\setminus i\ell_+$  as required, because if we set $w=z/2\pi it$ then \[t\in i\ell_+ \iff w\in \bR_{<0}.\]
Similarly $\Psi_-(t)$ is holomorphic on $\bC^*\setminus i\ell_-$.
The
Euler reflection formula  gives
\[\Gamma(w)\cdot \Gamma(1-w) =\frac{\pi}{\sin(\pi w)},\quad w\in \bC\setminus\bZ.\]
Since $\Gamma(n)=(n-1)!$ for $n\in \bZ_{>0}$, it follows that  $\Gamma(w)$ is nowhere vanishing. The same is therefore true of  $\Lambda(w)$.
The reflection formula also implies that for $w\in \bC\setminus \bZ$
\[\Lambda(w)\cdot \Lambda(-w)=\frac{1}{2\pi}\cdot  \frac{\pi }{\sin(\pi w)}\cdot e^{-( (w-\half)\log(w)+(-w+\half)\log(-w))},\]
where we used $\Gamma(1-w)=(-w)\cdot \Gamma(-w)$.
For the principal branch of $\log$ we have \[\log(w)=\log(-w)\pm \pi i\text{ when }\pm \Im(w)>0.\]
Thus we conclude that
\begin{equation}
\label{luminy}\Lambda(w)\cdot \Lambda(-w)=\frac{i\cdot e^{-\pi i(w-\half)}} {e^{\pi iw}-e^{-\pi i w}}=(1-e^{2\pi iw})^{-1}\end{equation}
when $\Im(w)>0$.
Note that if $w=z/2\pi i t$ then \[ t\in \bH_{\ell_-}\iff \Im(w)>0,\]
so we get the second of the relations \eqref{rabbit}. The other follows in the same way.
Property (ii) is immediate from the Stirling expansion \eqref{stirling}. Property (iii) is a simple consequence of the fact that $\Gamma(w)$ has a simple pole at $w=0$. Finally, the uniqueness statement follows from Lemma \ref{brownpaperbag}.
\end{proof}

\subsection{The finite uncoupled case}

In the case of a finite, uncoupled,  integral BPS structure we can construct a unique solution to the Riemann-Hilbert problem by superposing the solutions from the previous section.

\begin{thm}
\label{hd}
Let $(Z,\Gamma,\Omega)$ be a finite, uncoupled,  integral BPS structure. Suppose that $\xi\in \bT$ satisfies $\xi(\gamma)=1$ for all active classes $\gamma\in \Gamma$. Then the  corresponding Riemann-Hilbert problem has a unique solution, which associates to a non-active ray $\rr$ the function
\begin{equation}
\label{blbl}
\Psi_{\rr,\beta}(t)= \prod_{Z(\gamma)\in i\bH_\rr} \Lambda\bigg(\frac{Z(\gamma)}{2\pi i t}\bigg)^{\Omega(\gamma) \<\beta,\gamma\>}
\end{equation}
 where the product is taken over the finitely many active classes $\gamma\in \Gamma$ for which  $Z(\gamma)\in i\bH_\rr$.
\end{thm}

\begin{proof}
Note first that the expression \eqref{blbl} is holomorphic and non-zero on $\bH_\rr$ because in each factor both $Z(\gamma)/i$ and $t$ lie in $\bH_\rr$, so the argument of $\Lambda$ does not lie in $\bR_{<0}$.  The properties (RH2) and (RH3) then follow immediately as in the proof of Proposition \ref{twotier}. 
Consider two non-active rays $\rr_-$ and $\rr_+$ obtained by small  perturbations, anti-clockwise and clockwise respectively, of an  active ray $\ell$. Then $\ell\subset i\bH_{\rr_+}$ whereas $-\ell\subset i\bH_{\rr_-}$. Assume that  $t\in \bH_{\rr_-}\cap \bH_{\rr_+}$. Then  $t\in \bH_\ell$ and hence $-Z(\gamma)/2\pi it$ lies in the upper half-plane whenever $Z(\gamma)\in \ell$. Using \eqref{luminy} we therefore obtain 
\[\frac{\Phi_{\rr_+,\beta}(t)}{\Phi_{\rr_-,\beta}(t)}=\prod_{Z(\gamma)\in \ell} \bigg(\Lambda\Big(\frac{Z(\gamma)}{2\pi i t}\Big)\cdot\Lambda\Big(\frac{-Z(\gamma)}{2\pi i t}\Big)\bigg)^{\Omega(\gamma)\<\beta,\gamma\>}= \prod_{Z(\gamma)\in \ell} (1-e^{-Z(\gamma)/t} )^{\Omega(\gamma) \<\gamma,\beta\> }\]
 This is precisely the condition \eqref{garage} since  Lemma \ref{null} and the assumption on $\xi$ implies that
$\Phi_{\gamma}(t)=\exp(-Z(\gamma)/t)$ whenever $\Omega(\gamma)\neq 0$. 
\end{proof}

Using the Stirling expansion we obtain an asymptotic expansion
\begin{equation}\log \Phi_{\rr,\beta}(t)\sim \frac{Z(\beta)}{t}-2\pi i \theta(\beta)+\sum_{g\geq 1} \sum_{\gamma\in \Gamma}
 \frac{\Omega(\gamma)\<\beta,\gamma\>B_{2g}}{4g\,(2g-1)}\bigg(\frac{2\pi i t}{Z(\gamma)}\bigg)^{2g-1},\end{equation}
valid as $t\to 0$ in the half-plane $\bH_\rr$. We have set $\xi(\beta)=\exp(2\pi i \theta(\beta))$. The extra factor of 2 in the denominator arises because  in \eqref{blbl} we take a product over half the classes in $\Gamma$.  Note that this  expansion is independent of the choice of  ray $\rr\subset \bC^*$.

\subsection{Tau function in the uncoupled case}

Consider the expression
\[\Upsilon(w)=\frac{e^{-\zeta'(-1)}\, e^{\frac{3}{4}w^2} \,G(w+1)}{(2\pi)^{\frac{w}{2}}\,w^{\frac{w^2}{2}}},\]
where $G$ is the Barnes $G$-function  \cite{Barnes}, \cite[Appendix]{Vo}, and $\zeta(s)$ is the Riemann zeta function. It defines a holomorphic and nowhere vanishing function on $\bC^*\setminus \bR_{<0}$ which we specify uniquely by defining the factor $w^{w^2/2}$ using the principal value of $\log$.   The asymptotic expansion of $\Upsilon(w)$ is
\begin{equation}
\label{deb}\log \Upsilon(w)\sim -\frac{1}{12} \log(w) + \sum_{g\geq 2} \frac{B_{2g}}{2g(2g-2)} w^{-(2g-2)},\end{equation}
valid as $w\to \infty$ in the complement of any sector in $\bC^*$ containing the ray $\bR_{<0}$. This can be found for example in  \cite[Section 15]{Barnes} or \cite[Appendix]{Vo},
although note that Barnes uses a different indexing for the Bernoulli numbers, and refers to the real number
\[A=e^{\frac{1}{12}} \cdot e^{-\zeta'(1)}\]
 as the Glaisher-Kinkelin constant.

\begin{lemma}
\label{oldy}
There is an identity
\[\frac{d}{dw} \log \Upsilon(w)=w\frac{d}{dw} \log \Lambda(w).\]
\end{lemma}

\begin{proof}
Note that this is obvious at the level of the  asymptotic expansions. For the proof we use the identity
\[\frac{d}{dw} \log G(w+1)=  \frac{1}{2} \log(2\pi)+\frac{1}{2} - w+w\frac{d}{dw}\log \Gamma(w)\]
which can be found in  \cite[Section 12]{Barnes}
(see also \cite[Formula (A.13)]{Vo}) to get
\[\frac{d}{dw} \log \Upsilon(w)=\frac{1}{2} +w\frac{d}{dw}\log \Gamma(w)-w\log w.\]
The result then follows from \eqref{cheedale}  by taking $\log$ and differentiating.
\end{proof}

In the case of variations of BPS structures satisfying the conditions of Theorem \ref{hd} the following result gives a natural choice of $\tau$-function. 
\begin{thm}
Let $(\Gamma,Z_p,\Omega_p)$ be a framed, miniversal variation of finite, uncoupled,  integral BPS structures over a complex manifold $M$. Given a ray $\rr\subset \bC^*$, the function
\begin{equation}\label{holo}\tau_\rr(p;t)=\prod_{Z_p(\gamma)\in i\bH_\rr} \Upsilon\bigg(\frac{Z(\gamma)}{2\pi i t}\bigg)^{\Omega(\gamma)}\end{equation}
is a   $\tau$-function for the  family of solutions of Theorem \ref{hd}. 
\end{thm}

\begin{proof}
The expression \eqref{holo} is holomorphic and non-zero on $\bH_\rr$ for the same reason given  in the proof of Theorem \ref{hd}. It is also clearly invariant under simultaneous rescaling of $Z$ and $t$.
Choosing a basis $(\gamma_1,\cdots,\gamma_n)\subset \Gamma$, and using the local co-ordinates $z_i=Z(\gamma_i)$ on $M$, we have
\[2\pi i \cdot \sum_i \epsilon_{ij}\frac{\partial \log \tau_\rr}{\partial z_i}=\sum_i   \sum_{Z_p(\gamma)\in i\bH_\rr} \Omega(\gamma) \, \frac{\epsilon_{ij}\,m_i(\gamma)}{ t} \Big(\frac{d}{dw}\Big|_{w=\frac{Z(\gamma)}{2\pi i t}}\Big)\log \Upsilon (w),\]
where we wrote $\epsilon_{ij}=\<\gamma_i,\gamma_j\>$ and   used the decomposition $\gamma=\sum_i m_i(\gamma)\gamma_i$. Using Lemma \ref{oldy} this  can be rewritten as
\[\sum_{Z_p(\gamma)\in i\bH_\rr} \Omega(\gamma)\, \frac{\<\gamma,\gamma_j\>}{ t} \Big(w\frac{d}{dw}\Big|_{w=\frac{Z(\gamma)}{2\pi i t}}\Big)\log \Lambda (w)=\frac{\partial}{\partial t} \sum_{Z_p(\gamma)\in i\bH_\rr} 
\log \Lambda\Big(\frac{Z(\gamma)}{2\pi i t}\Big)^{\Omega(\gamma) \<\gamma_j,\gamma\>}=\frac{\partial\log \Psi_{\rr,\gamma_j}}{\partial t},\]
which gives \eqref{nick2}, and hence completes the proof that \eqref{holo} defines a $\tau$-function.
\end{proof}

Applying \eqref{deb} we get an asymptotic expansion
\begin{equation}\label{em}\log \tau_\rr(p;t)\sim \frac{1}{24} \sum_{\gamma\in \Gamma} \Omega(\gamma)\log \bigg(\frac{2\pi i t}{Z(\gamma)}\bigg) +\sum_{g\geq 2} \sum_{\gamma\in \Gamma}   \frac{\Omega(\gamma)\cdot B_{2g}}{4g\, (2g-2)} \bigg(\frac{2\pi i t}{Z(\gamma)}\bigg)^{2g-2}\end{equation}
valid as $t\to 0$ in  the half-plane $\bH_\rr$. Once again this expansion is independent of the ray $r\subset \bC^*$.

\begin{remark}
\label{andrei} The function $\Upsilon(w)$  is closely related to the function $\gamma_\hbar(x;\Lambda)$ which plays an important role in Okounkov and Nekrasov's work on supersymmetric gauge theories \cite[Appendix A]{ON}. More precisely,  they consider a function $\gamma_\hbar(x;\Lambda)$ which is uniquely defined up to a linear function in $x$ by two properties: a difference equation, and the existence of an asymptotic expansion. Using the property $G(w+1)=\Gamma(w)\cdot G(w)$ of the Barnes $G$-function, and the expansion \eqref{deb},  it follows that we can take
\[e^{\gamma_t(z;1)}={e^{-\zeta'(-1)}\cdot G\Big(\frac{z}{t}+1\Big)\cdot t^{\frac{1}{2} (\frac{z}{t})^2-\frac{1}{12} }}\cdot {(2\pi)^{-\frac{z}{2t}}}=\Upsilon\Big(\frac{z}{t}\Big)\cdot t^{\frac{1}{12}}\cdot e^{\frac{1}{t^2}(\frac{1}{2}z^2\log(z)-\frac{3}{4} z^2)}.\]
 The term appearing in the exponential on the right-hand side gives rise to the prepotential of the gauge theory. We hope that a clearer understanding of the definition of the $\tau$-function will enable us to give a mathematical definition of the prepotential. 
\end{remark}


\section{Geometric case: Gromov-Witten invariants}
\label{gw}

In this section we consider a class of BPS structures related to closed topological string theory  on a compact Calabi-Yau threefold. In mathematical terms they arise from stability conditions on the category of coherent sheaves supported in dimension $\leq 1$. These BPS structures are uncoupled but not finite. We will show that  formally applying the expression \eqref{em}  in this case reproduces the genus 0 degenerate contributions to the Gromov-Witten generating function.
In \cite{tocome} we give a more careful analysis for the special case of the resolved conifold.

\subsection{Gopakumar-Vafa invariants}
 Let $X$ be a smooth projective Calabi-Yau threefold. For the sake of notational simplicity we will assume that the group $H_2(X,\bZ)$ is torsion-free. The Gromov-Witten potential  of $X$ is a formal series
\begin{equation}
\label{help}F(x,\lambda)=\sum_{g\geq 0}\sum_{\beta\in H_2(X,\bZ)} \GW(g,\beta) \cdot x^\beta \cdot \lambda^{2g-2},\end{equation}
where $\GW(g,\beta)\in \bQ$ is the genus $g$ Gromov-Witten invariant for stable maps  of class $\beta\in H_2(X,\bZ)$. Note that by definition these invariants are nonzero only for  effective curve classes $\beta\geq 0$. The symbols $x^\beta$ are  formal variables  living in a suitable completion of the effective cone in the group ring of $H_2(X,\bZ)$, and  $\lambda$ is a formal parameter corresponding to the string coupling.

We can split the series \eqref{help} into contributions from constant and non-constant maps
\[F(x,\lambda)=F_0(x,\lambda)+F'(x,\lambda).\]
The contribution from the constant maps  \cite[Theorem 4]{FP} is
\begin{equation}
\label{cons}F_0(x,\lambda)=a_0(x) \lambda^{-2} +a_1(x) +\chi(X)\cdot \sum_{g\geq 2}  \frac{(-1)^{g-1}\cdot B_{2g}\cdot B_{2g-2}}{4g\cdot (2g-2)\cdot (2g-2)!}\lambda^{2g-2},\end{equation}
where the expressions $a_0(x)$ and $a_1(x)$ can be found for example in \cite{P}. Although the precise form of these expressions  will not be relevant here it is worth noting  that, unlike the higher genus terms, they involve the variables $x^\beta$.
Turning now to the contributions from non-constant maps, the Gopakumar-Vafa conjecture \cite{GV, P} claims that there exist integers  $\GV(g,\beta)\in \bZ$, such that
\[F'(x,\lambda)=\sum_{\beta> 0}\sum_{g\geq 0} \GV(g,\beta)\sum_{k\geq 1} \frac{1}{k} \,\Big(2\sin \Big( \frac{k\lambda}{2} \Big)\Big)^{2g-2} \, x^{k\beta}.\]

We will be particularly interested in the expression
\begin{equation}
\label{yp}\sum_{\beta>0}\GV(0,\beta)\sum_{k\geq 1} \frac{1}{k} \, \Big(2\sin \Big( \frac{k\lambda}{2} \Big)\Big)^{-2} \, x^{k\beta},\end{equation}
which gives the
contribution from genus 0 Gopakumar-Vafa invariants.
Note that using the Laurent expansion
\[(2\sin(s/2))^{-2} = \frac{1}{s^2}-\frac{1}{12}+ \sum_{g\geq 2} \frac{(-1)^{g-1} B_{2g}}{2g(2g-2)!} \cdot  s^{2g-2},\]
we can write the coefficient of $\lambda^{2g-2}$ in \eqref{yp} as
\begin{equation}
\label{blag}
\sum_{\beta>0} \GV(0,\beta) \cdot \frac{(-1)^{g-1} B_{2g}}{2g(2g-2)!} \cdot  \Li_{3-2g}(x^\beta),\end{equation}
at least  for $g\geq 2$.

\subsection{Torsion sheaf BPS invariants}

 Let $\cA=\Coh_{\leq 1}(X)$ denote the abelian category of coherent sheaves on $X$ supported in dimension $\leq 1$.  Any  sheaf $E\in \cA$ has a Chern character $\ch(E)\in H^*(X,\bZ)$, which via Poincar{\'e} duality we can view as an element
\[\ch(E)=(\beta,n)\in \Gamma=H_2(X,\bZ)\oplus \bZ.\]
We note that for any objects $E_1,E_2$ of $\cA$, the Riemann-Roch theorem tells us that
\[\chi(E_1,E_2)=\sum_{i\in \bZ} (-1)^i \dim_\bC \Ext^i_X(E_1,E_2)=\int_X \ch(E_1)^{\dual}\ch(F)\td(X) =0,\]
because the intersection number of any two curves on a threefold is zero. We therefore take $\<-,-\>$ to be the zero form  on  $\Gamma$. We define a central charge
$Z\colon \Gamma\to \bC$ via the formula \[Z(\beta,n)=2\pi(\beta\cdot \omega_\bC-n),\]
where $\omega_\bC=B+i\omega\in H^2(X,\bC)$ is a complexified K{\"a}hler class.
The assumption that $\omega$ is  K{\"a}hler  ensures that for any nonzero object $E\in \cA$ the complex number $Z(E)\in \bC$ lies in the semi-closed upper half-plane
\[\bar{\cH}= \{z=r\exp(i\pi \phi): r\in \bR_{>0}\text{ and }0<\phi\leq 1\}\subset \bC^*.\]
It follows that $Z$  defines a stability condition on the abelian category $\cA$.
 
For each class $\gamma\in \Gamma$ there is an associated BPS invariant
$\Omega(\gamma)\in \bQ$  first constructed by Joyce and Song (\cite{JS}, see particularly Sections 6.3--6.4). They  are defined using moduli stacks of semistable objects in $\cA$, and should not be confused with the ideal sheaf curve-counting invariants appearing in the famous MNOP conjectures \cite{MNOP}.  Joyce and Song prove that the numbers $\Omega(\gamma)$  are independent of the complexified K{\"a}hler class $\omega_\bC$. This is to be expected, since wall-crossing is trivial when the form $\<-,-\>$ vanishes: see Remark \ref{abelian} below.

A direct calculation \cite[Section 6.3]{JS} shows that
 \begin{equation}\label{montreux}\Omega(0,n)=-\chi(X),\qquad n\in \bZ\setminus\{0\}.\end{equation}
It is expected \cite[Conjecture 6.20]{JS} that when $\beta>0$ is a positive curve class
\begin{equation}
\label{lausanne}
\Omega(\pm \beta,n)=\GV(0,\beta),\end{equation}
and, in particular, is independent of $n$.  We shall assume this in what follows.
 We emphasise that the higher genus Gopakumar-Vafa invariants are invisible from the point-of-view of the  torsion sheaf invariants $\Omega(\gamma)$.

\subsection{Formal computation of the $\tau$-function}

The discussion of the last subsection gives rise to a framed variation of BPS structures  over the complexified K{\"a}hler cone of  $X$, in which the  BPS invariants $\Omega(\gamma)$ are constant.   Let us consider the family of double  structures as defined in Section \ref{double}.
 We can identify
\[\Gamma\oplus\Gamma^\vee=H_*(X,\bZ)\]
equipped with the standard intersection form, and the resulting BPS structures are all uncoupled. 

\begin{remarks}
\begin{itemize}
\item[(i)] We can easily extend this to a miniversal variation if required, by first introducing an extra factor $q\in \bC^*$ rescaling the central charge $Z$, and also adding a component $Z^\dual\colon \Gamma^\vee \to \bC$ as in Section \ref{double}. The resulting central charge is
\[Z((\beta,n),\lambda)=2\pi q(\beta\cdot \omega_\bC-n)+Z^\vee(\lambda).\]
Nothing interesting is gained by doing this however.\smallskip

\item[(ii)]
We view the doubled structures defined here as an approximation to the correct BPS structures, which should also incorporate BPS invariants corresponding to objects of the full derived category $\D^b\Coh(X)$ supported in all dimensions. To define these rigorously would involve constructing stability conditions on $\D^b\Coh(X)$, which for $X$ a general compact Calabi-Yau threefold is a well-known unsolved problem (see \cite{BMS} and \cite{PT} for more on this).
\end{itemize}
\end{remarks}

Although our BPS structures are not finite, we can nonetheless try to solve the associated Riemann-Hilbert problem by superposing infinitely many gamma functions. Let us formally apply \eqref{em} to compute for $g\geq 2$ the coefficient 
\[\sum_{\gamma\in \Gamma}   \frac{\Omega(\gamma)\cdot B_{2g}}{4g(2g-2) } \cdot \frac{1}{Z(\gamma)^{2g-2}}\]
of $(2\pi it)^{2g-2}$ in the asymptotic expansion of the $\tau$-function.\footnote{This calculation was worked out jointly with K. Iwaki.}
The contribution from zero-dimensional sheaves is
\[ - \frac{\chi(X) B_{2g}}{2g(2g-2) (2\pi)^{2g-2}} \cdot\sum_{k\geq 1} \frac{1}{ k^{2g-2}}=
\frac{\chi(X) (-1)^{g-1}B_{2g} \,B_{2g-2} }{4g (2g-2)(2g-2)!},\]
which agrees with \eqref{cons}. The contribution from one-dimensional sheaves is
\begin{equation}\label{blox}\frac{B_{2g}}{2g(2g-2) (2\pi)^{2g-2}} \sum_{\beta>0}\sum_{k\in \bZ}  \frac{\GV(0,\beta) }{    {(v_\beta-k)}^{2g-2}},\end{equation}
where $v_\beta=\omega_\bC\cdot \beta$.
Using the identity
\[ \sum_{k\in \bZ} \frac{1}{(z-k)^{2g-2}}= \frac {(2\pi i)^{2g-2}}{(2g-3)!} \Li_{3-2g}(e^{2\pi i z}),\]
valid for $\Im(z)>0$ and $g\geq 2$, 
we can rewrite \eqref{blox} as
\begin{equation}
\label{greg}
\frac{(-1)^{g-1} B_{2g} }{2g(2g-2)!} \sum_{\beta>0}\GV(0,\beta)\cdot \Li_{3-2g}(e^{2\pi i v_\beta}),\end{equation}
which then agrees with \eqref{blag}.
We conclude that under the variable change
\[\lambda=2\pi it, \quad x_\beta=\exp(2\pi i v_\beta),\]
the $\log$ of the $\tau$-function reproduces the genus 0 degenerate contributions to the Gromov-Witten generating function \eqref{help},
at least for positive powers of $\lambda$. 

\begin{remark}
In the paper \cite{tocome} we give a rigorous solution to the Riemann-Hilbert problem in the case when $X$ is the resolved conifold. This involves writing down a non-perturbative function which has the above asymptotic expansion. \end{remark}


\section{Quadratic differentials and exact WKB analysis}
\label{in}

The only examples of CY$_3$ categories where the full stability space is understood come from quivers with potential associated to triangulated surfaces. The associated stability spaces can be identified with moduli spaces of meromorphic quadratic differentials on Riemann surfaces \cite{BS}, and the associated BPS invariants then count finite-length trajectories of these differentials. It turns out that the corresponding Riemann-Hilbert problems are closely related to the exact WKB analysis of  time-independent Schr{\"o}dinger equations. We give a brief and sketchy treatment of this connection here; we hope to return to this subject in future papers.

\subsection{Quadratic differentials}

For more details on the contents of this section see \cite{BS}.
Let us start by fixing data
\[g\geq 0, \quad m=\{m_1,\cdots,m_k\}, \quad k\geq 1,\  m_i\geq 2,\]
and  consider the space
$\Quad(g,m)$ consisting of equivalence classes of pairs $(S,q)$, where $S$ is a compact Riemann surface of genus $g$ and $q$ a meromorphic quadratic differential on $S$ with simple zeroes, and poles of multiplicities $m_i$.
It is a complex orbifold of dimension
\[n=6g-6+\sum_i (m_i+1).\]

Associated to a point $(S,q)$ is a double cover
$\pi\colon \Hat{S}\to S$
branched at the zeroes and odd-order poles of $q$. We denote by $\hS^\circ\subset \hS$ the complement of the inverse image of the poles of $q$, and define the hat-homology group
 \[\Gamma=H_1(\hS^{\circ};\bZ)^-, \]
where the superscript denotes the $-1$ eigenspace under the action of the covering involution of $\pi\colon \Hat{S}\to S$. The intersection form defines a skew-symmetric form
\[\<-,-\>\colon \Gamma \times \Gamma \to \bZ.\]
The groups $\Gamma$ form a local system over $\Quad(g,m)$.

The meromorphic abelian differential $\sqrt{q}$ is well-defined on the double cover $\hS$, and holomorphic on $\hS^\circ$. It defines a de Rham cohomology class in $H^1(\hS^\circ;\bC)^-$ and can be viewed as a group homomorphism $Z\colon \Gamma\to \bC$
\[Z(\gamma)=\int_\gamma \sqrt{q}.\]
It was proved in \cite{BS} that the period map
\[\pi\colon \Quad(g,m) \to \Hom_\bZ(\Gamma,\bC)\isom \bC^n.\]
is a local analytic isomorphism.

By a trajectory of a differential $(S,q)$ we mean a path in $S$ along which $\sqrt{q}$ has constant phase $\theta$. A finite-length trajectory is of one of two types:
\begin{itemize}
\item[(a)] a saddle connection connects two zeroes of the differential (not necessarily distinct);
\item[(b)] a closed trajectory: any such moves in an annulus of  trajectories called a ring domain.
\end{itemize}
Any finite-length trajectory  can be lifted to a closed cycle in $\hS$ which defines an associated class $\gamma\in \Gamma$. All trajectories in a ring domain have the same class, so we can also talk about the class associated to the ring domain.

 Let us assume that our quadratic differential is generic in the sense that if $\gamma_1,\gamma_2$ are two finite-length trajectories of the same phase, then their classes are  proportional in $\Gamma$.
We then define the BPS invariants of $q$ by
\[\Omega(\gamma)=\#\big\{\mbox{saddle connections of class $\gamma$}\big\} - 2\cdot \#\big\{\mbox{ring domains of class $\gamma$}\big\}.\]
The reason for the coefficient $-2$ is that a ring domain leads to a moduli space of stable objects isomorphic to $\bP^1$.  See \cite[Theorem 1.4]{BS}. In physical terms saddle connections represent hypermultiplets, whereas ring domains represent vector multiplets.

\begin{claim}
The data $(\Gamma, Z,\Omega)$ described above defines a miniversal variation of  convergent BPS structures over the orbifold $\Quad(g,m)$.
\end{claim}

\noindent{\it Sketch proof.}
To check  the wall-crossing formula one can use the results of \cite{BS} to view $\Quad(g,m)$ as an open subset of a space of stability conditions on a CY$_3$ triangulated category defined by a quiver with potential, and then apply the theory of wall-crossing for generalised  DT invariants \cite{JS,KS1}.  

The fact that the BPS structures are convergent should follow from the results of \cite[Section 5]{BS}. The basic point is that the only non-finiteness in the BPS spectrum arises from finitely many ring domains. Each of these contributes an infinite collection of saddle connections with classes of the form $\gamma+n\alpha$, where $\alpha$ is the class of the ring domain. But for sufficiently large $R>0$ these saddle connections give a finite contribution to the sum \eqref{concon}.
\qed \smallskip

\subsection{Voros symbols}
The weak Riemann-Hilbert problem (see Remark \ref{oh}(iii)) defined by the above BPS structures  can be solved using exact WKB analysis of an associated Schr{\"o}dinger equation.  This was  essentially proved by Iwaki and Nakanishi \cite{IN} following Gaiotto, Moore and Neitzke \cite{GMN2}.

Suppose given a quadratic differential $(S,q)$ as above. Let us also choose 
 a projective structure  on the Riemann surface $S$.\footnote{ If the quadratic differential $q$ has a double pole at a point $p\in S$ one should assume that this projective structure (represented in \cite{IN} by the term $Q_2(z)$) also has a pole of a particular form \cite[Assumption 2.5]{IN}.}
We can then invariantly consider the holomorphic Schr{\"o}dinger equation
\[\hbar^2 \frac{d^2}{dz^2} y(z,\hbar) = q(z)\cdot y(z,\hbar),\]
where $z$ is a co-ordinate in the chosen projective structure.
The WKB method involves substituting
\[y(z_1,\hbar)=\exp\Big( \int_{z_0}^{z_1} T(z)\,dz \Big), \qquad T(z)=\sum_{k\geq 0} T_k(z) \cdot \hbar^{k-1},\]
and then solving for $T_j(z)$ order by order. This gives rise to a recursion
\[\frac{dT_{k-1}}{dz}+\sum_{i+j=k} T_i(z) T_j(z)=0,\quad k>0,\]
together with the initial condition $T_{0}(z)^2 =q(z)$. Depending on the choice of square-root taken to define $T_0(z)$ one then obtains two systems of solutions $T_k^{\pm}(z)$. The differences
\[\omega_k(z)=\frac{1}{2} \big(T^+_k(z) -T^-_k(z)\big) dz\]
are single-valued meromorphic one-forms on the spectral cover $\hS$, which vanish unless $k$ is even. One has $\omega_0=\sqrt{q(z)} \, dz$.

The formal cycle Voros symbol associated to a class $\gamma\in \Gamma$ is the formal sum
\begin{equation}\exp(V_\gamma)= \exp\Big(\frac{1}{\hbar} \int_\gamma \omega_0\Big)\cdot \exp\Big(\sum_{g\geq 1} \hbar^{2g-1} \int_\gamma \omega_{2g} \Big).\end{equation}

Iwaki and Nakanishi, relying on analytic results of \cite{KoSc},  show that if $\rr\subset \bC^*$ is a non-active ray for the BPS structure defined by $(S,q)$, then the sum over $g\geq 1$ in the above formal expressions can be Borel summed in the direction $\rr$. This results in Voros symbols which are holomorphic functions of $t=\hbar$ defined in a neighbourhood of $0$ in the half-plane $\bH_\rr$. They also compute the wall-crossing behaviour for these Borel sums as one varies the active ray $\rr\subset \bC^*$.

\begin{claim}
Given a quadratic differential $(S,q)$ as above,  the Borel sums of the cycle Voros symbols  give a solution to the corresponding weak Riemann-Hilbert problem  with $t=\hbar$.
\end{claim}
\smallskip

\noindent{\it Sketch proof.}
For the definition of the weak Riemann-Hilbert problem see Remark \ref{oh}(iii). The claim should follow from  the work of Iwaki and Nakanishi \cite{IN}, although a careful proof would require additional continuity arguments to take care of differentials with multiple saddle connections. See particularly Theorem 2.18, Theorem 3.4 and formula (2.21). \qed
\smallskip

It is interesting to ask whether given a suitable choice of base projective structure, the Voros symbols in fact give solutions to the full Riemann-Hilbert problem as stated in Section \ref{real}.  Another interesting topic for further research is the connection with topological recursion \cite{ey}, which is known to be closely related to exact WKB analysis. In particular, it is interesting to ask whether the $\tau$-function computed by topological recursion gives a $\tau$-function in the sense of this paper.

\begin{appendix}


\section{Variations of BPS structure}
\label{two}

For BPS structures satisfying the conditions of Proposition \ref{barneyy}, the BPS automorphisms $\bS(\ell)$ can be realised   as birational automorphisms of the twisted torus $\bT$. In general however, even for BPS structures coming from quivers with potential, this is not the case.  Thus we need some other approach to defining a variation of BPS structures. If we want to avoid making unnecessary assumptions, the only way to proceed is via formal completions of the ring $\bC[\bT]$. In this section we give a rigorous definition along these lines following Kontsevich and Soibelman \cite[Section 2]{KS1}.

\subsection{Introductory remarks}

Before starting formal definitions in the next subsection we  consider here a slightly simplified situation which should help to make the general picture clearer. Let $(\Gamma,Z,\Omega)$ be a BPS structure. Given a basis $(\gamma_1,\cdots,\gamma_n)\subset \Gamma$ there is a corresponding cone
\[\Gamma^\oplus=\big\{\gamma =\sum  d_i \gamma_i \text{ with }  d_i\in \bZ_{\geq 0}\big\}\subset \Gamma,\]
whose elements we call positive. We call a class $\gamma\in \Gamma$ negative if $-\gamma$ is positive.

Let us  assume that we can find a basis such that every active class  is either positive or negative, and further that all positive classes $\gamma\in \Gamma$ satisfy $\Im Z(\gamma)>0$. These assumptions are always satisfied  for BPS structures arising from stability conditions on quivers.
 Let us also temporarily ignore the difference between the tori $\bT_\pm$.

The co-ordinate ring $\bC[\bT]$ can be identified with the algebra of Laurent polynomials
\[\bC[\bT]=\bC[y_1^{\pm 1}, \cdots, y_n^{\pm n}],\]
and a regular automorphism of $\bT$ corresponds to an automorphism of this algebra. The element $\DT(\ell)$ corresponding to an active ray $\ell\subset \bC^*$ does not  in general define an element of $\bC[\bT]$ however,  since it could an infinite sum of characters.
If a ray $\ell\subset \bC^*$ lies in the upper half-plane, the series $\DT(\ell)$  defines an element of the power series ring
\[\bC[[\bT]]=\bC[[y_1, \cdots, y_n]],\]
which is a Poisson algebra in  the same way as $\bC[\bT]$. 
We can then  define an algebra automorphism
\[\bS(\ell)^*=\exp \{\DT(\ell),-\}\in \Aut  \bC[[\bT]],\]
and  interpret the wall-crossing formula \eqref{wcw} in the group of  automorphisms of  $\bC[[\bT]]$, at least for sectors $\Delta\subset \bC^*$ contained in the upper half-plane.

In what follows we will modify this  procedure  in two ways. Firstly, to avoid the assumption on active classes we will work with completions of $\bC[\bT]$ defined by more general cones in $\Gamma$. Secondly, note that if the form $\<-,-\>=0$ is trivial, the  automorphisms $\bS(\ell)^*$ as defined above will all be identity maps, and the wall-crossing formula will become vacuous. One way to  avoid this problem would be to force the form $\<-,-\>$ to be non-degenerate by  passing to the double BPS structure as in Section \ref{double}.  Instead, we will  formulate the wall-crossing formula in a group defined abstractly by using the Baker-Campbell-Hausdorff formula to formally exponentiate elements of $\bC[[\bT]]$.  

\subsection{Formal completions}
Let $(\Gamma,Z,\Omega)$ be a BPS structure. To an acute sector   $\Delta\subset \bC^*$  we associate a Poisson subalgebra
\begin{equation}
\label{frioff}\bC_\Delta[\bT]=\bigoplus_{Z(\gamma) \in \Delta\sqcup\{0\}}  \bC\cdot x_\gamma\subset \bC[\bT].\end{equation}
We will now introduce a natural completion of this algebra.

Define the height of an element $a=\sum_\gamma  a_\gamma\cdot x_\gamma$   to be
\[\height(a)=\inf\{|Z(\gamma)|: \gamma \in \Gamma\text{ such that }a_\gamma\neq 0\}.\]
Note that since $\Delta$ is acute, we have
\[|Z(\gamma_1+\gamma_2)|\geq \min (|Z(\gamma_1)|,|Z(\gamma_2)|),\]
whenever $\gamma_i\in \Gamma$ satisfy $Z(\gamma_i)\in \Delta$. It follows that \[\height(a\cdot b)\geq \max(\height(a),\height(b)),\quad \height(\{a,b\})\geq \max(\height(a),\height(b)).\]
In particular, for any $N>0$, the subspace \[\bC_\Delta[\bT]_{\geq N}\subset \bC_\Delta[\bT]\] consisting of elements of height $\geq N$ is a Poisson ideal. We can therefore consider the inverse limit of the quotient Poisson algebras as $N\to \infty$:
\[\bC_{\Delta}[[\bT]]=\varprojlim_{N} \,\bC_{\Delta}[\bT]_{< N}, \qquad  \bC_{\Delta}[\bT]_{< N}= \bC_{\Delta}[\bT]/\bC_{\Delta}[\bT]_{\geq N}.\] 
The resulting completion $\bC_{\Delta}[[\bT]]$
can be identified with  the set of formal sums \begin{equation}\label{lon} \big.\sum_{Z(\gamma)\in \Delta} a_\gamma \cdot x_\gamma,\end{equation} such that for any $N>0$ there are only finitely many terms with $|Z(\gamma)|<N$. We define the height of such an element to be  the minimum value of $|Z(\gamma)|$ occurring.

\subsection{Lie algebra and associated group}
Continue with the notation from the last subsection. For each $N>0$, 
the Poisson bracket induces the structure of a nilpotent  Lie algebra  on the Poisson ideal
\[\fg_{\Delta,\leq N} \subset \bC_{\Delta}[\bT]_{\leq N}\]
consisting of elements of positive height. The Baker-Campbell-Hausdorff formula then gives rise to a unipotent algebraic group $G_{\Delta,\leq N}$ with a bijective exponential map
\[\exp\colon \fg_{\Delta,\leq N} \to G_{\Delta,\leq N}.\]
Taking the limit as $N\to \infty$ we obtain a pro-nilpotent Lie algebra and a pro-unipotent group
related by a bijective exponential map
\[\exp\colon \Hat{\fg}_{\Delta}\to \Hat{G}_{\Delta}.\]

More concretely, the Lie algebra $\Hat{\fg}_\Delta$ consists of formal sums \eqref{lon} of positive height, viewed as a Lie algebra via the Poisson bracket. The group \[\Hat{G}_\Delta=\{\exp(x):x\in \Hat{\fg}_\Delta\},\]
consists of formal symbols $\exp(x)$ for elements $x$ of the Lie algebra. The group structure is  defined using the Baker-Campbell-Hausdorff formula:
\[\exp(x)\cdot \exp(y) = \exp \Big( (x+y)+\frac{1}{2}[x,y]+\cdots \Big).\]
Note that if $\Delta_1\subset \Delta_2$ are nested acute sectors then there is an obvious embedding of Lie algebras ${\fg}_{\Delta_1}\subset {\fg}_{\Delta_2}$, which induces an injective homomorphism $\Hat{G}_{\Delta_1}\into \Hat{G}_{\Delta_2}$ of the corresponding groups. 
We happily confuse elements of $\Hat{G}_{\Delta_1}$ with the corresponding elements of $\Hat{G}_{\Delta_1}$.

 \begin{remark}The Lie algebra $\Hat{\fg}_\Delta$ acts on the algebra $\bC_{\Delta}[[\bT]]$ via Poisson derivations: 
\[x\cdot y = \{x,y\}, \quad x\in \Hat{\fg}_\Delta, \quad y\in \bC_{\Delta}[[\bT]].\]
This exponentiates to give an action of $\Hat{G}_\Delta$ by Poisson automorphisms:
\begin{equation}
\label{rats}\exp(x)\cdot y = \exp \{x, - \} (y), \quad x\in \Hat{\fg}_\Delta, \quad y\in \bC_{\Delta}[[\bT]].\end{equation}
On the left-hand side of \eqref{rats} the symbol $\exp$ is purely formal as discussed above, whereas on the right-hand side it denotes the exponential of the Poisson derivation $\{x,-\}$, viewed as an endomorphism of the vector space $\bC_{\Delta}[[\bT]]$.\end{remark}

\subsection{Products over rays}

Let  $(\Gamma, Z, \Omega)$ be a BPS structure.  Recall the notion of the height of a ray from Section \ref{raydiagram}. Let us now also fix an acute sector $\Delta\subset \bC$.
The support property ensures that for any ray $\ell\subset \Delta$ the expression
\[\DT(\ell)=-\hspace{-.8em}\sum_{\gamma\in \Gamma: Z(\gamma)\in \ell} \DT(\gamma) \cdot x_\gamma\in \bC_\Delta[[\bT]]\]
defines an element of the Lie algebra $\Hat{\fg}_\Delta$. By definition, its height is equal to $\height(\ell)$. We denote the corresponding group element by
\[\bS(\ell)^*=\exp \DT(\ell)\in \Hat{G}_\Delta.\]

Given $N>0$, we can  consider the  truncation
\[\bS(\ell)^*_{<N}=\exp \DT(\ell)_{<N} \in \Hat{G}_{\Delta,<N}.\]
This element is non-trivial only for the finitely many rays of height $<N$. Therefore, we can form the finite clockwise product
\begin{equation}
\label{fin}\bS(\Delta)^*_{<N}=\prod_{\ell\subset \Delta} \;\bS(\ell)_{<N} \in \Hat{G}_{\Delta,<N}.\end{equation}
Taking the limit  $N\to \infty$ then gives a well-defined element 
\begin{equation}
\label{late}\bS(\Delta)^*=\prod_{\ell\subset \Delta} \;\bS(\ell)^* \in \Hat{G}_{\Delta}.\end{equation}
The product on the right hand side will usually be infinite.



\subsection{Deforming the central charge}
\label{works}

There is one more detail we have to deal with before we can give the definition of a variation of BPS structure. The problem is that the groups $\Hat{G}_\Delta$ and  $\Hat{G}_{\Delta,<N}$ which we defined above depend on the central charge of the BPS structure in a highly discontinuous way. This problem arises already in the definition \eqref{frioff}: the subalgebra $\bC_\Delta(\bT)$ will change whenever the central charge $Z(\gamma)$ of any class $\gamma\in \Gamma$ crosses a boundary ray of $\Delta$, and since this will happen on a dense subset of the set of possible central charges $Z\colon \Gamma\to \bC$ this leads to a highly discontinuous family of algebras.

To solve this problem, let us fix $C>0$ and consider the subset $\Gamma(\Delta,C)\subset \Gamma$  consisting of all non-negative integral combinations of elements $\gamma\in \Gamma$ which satisfy \begin{equation}\label{blog}Z(\gamma)\in \Delta, \qquad |Z(\gamma)|>C\cdot \|\gamma\|.\end{equation}
By definition  this is a submonoid of $\Gamma$ under addition, so
there is a  Poisson subalgebra \[\bC_{(\Delta,C)}[\bT]=\bigoplus_{\gamma\in \Gamma(\Delta,C)} \bC\cdot x_\gamma\subset \bC_\Delta[\bT].\]  We can then define Lie algebras and associated groups exactly as above. Moreover, if we take $C$ smaller than the constant in the support property for the BPS structure $(\Gamma,Z,\Omega)$ then all the identities of the last subsection will take place in these groups, since we will then have \[\Omega(\gamma)\neq 0\implies \gamma\in \Gamma(\Delta,C).\]

Consider now the quotient algebra
$\bC_{(\Delta,C)}[\bT]_{<N}$ defined as before. 
This depends only on the set of  $\gamma\in \Gamma(\Delta,C)$  for which $|Z(\gamma)|<N$, and hence only on the finite subset
\begin{equation}
\label{subsub}\big\{\gamma\in \Gamma: Z(\gamma)\in \Delta\text{ and } C \|\gamma\|<|Z(\gamma)|<N\big\}\subset \{\gamma\in \Gamma: \|\gamma\|<N/C\}\subset \Gamma.\end{equation}
Moreover, as $Z$ varies, it remains constant in the complement of the finite system of hypersurfaces. Clearly the same remarks apply to the corresponding group $\Hat{G}_{(\Delta,C),<N}$.

\subsection{Variations of BPS structure}

We can now give a rigorous  definition of a variation of BPS structures. The conditions are explained more fully below. 

\begin{defn}
\label{nick}
A variation of BPS structure consists of a complex manifold $M$, together with BPS structures $(\Gamma_p,Z_p,\Omega_p)$ indexed by the  points $p\in M$, such that \smallskip

\begin{itemize}
\item[(V1)] {\it Local system of charge lattices}. The charge lattices  $\Gamma_p$ form a local system of abelian groups, and the intersection form is covariantly constant.\smallskip

\item[(V2)] {\it Holomorphic variation of central charge.} Given a covariantly constant family of elements $\gamma_p\in \Gamma_p$, the central charges $Z_p(\gamma_p)\in \bC$ are holomorphic functions of $p\in M$.\smallskip

\item[(V3)] {\it Uniform support property.} Fix a covariantly constant family of norms \[\|\cdot\|_p\colon \Gamma_p\tensor_\bZ \bR\to \bR_{>0}.\] Then for any compact subset $F\subset M$ there is a  $C>0$ such that
\[\Omega_p(\gamma_p)\neq 0 \text{ for some }p\in F \implies |Z_p(\gamma_p)|> C\cdot \|\gamma\|_p.\]

\item[(V4)] {\it Wall-crossing formula.} Suppose given a contractible open subset $U\subset M$,   a constant $N>0$, and a convex sector $\Delta\subset \bC^*$. We can trivialise the local system $\Gamma_p$ over $U$ and hence identify $\Gamma_p$ with a fixed lattice $\Gamma$. Take $C>0$ as in (V3) and assume that the subset \eqref{subsub} is constant for  the BPS structures corresponding to points $p\in U$. Then the clockwise products 
\[\bS_p(\Delta)^*_{<N}=\prod_{\ell\subset \Delta} \;\bS_p(\ell)^*_{<N} \in \Hat{G}_{(\Delta,C),<N},\]
defined as in \eqref{fin}
 are  constant as $p\in U$ varies.
 \end{itemize}
\end{defn}

We say that the variation is framed if the local system of lattices $\Gamma_p$ over $M$ is  trivial. The lattices $\Gamma_p$ can then all be  identified with a fixed lattice   $\Gamma$, and we write the variation as $(\Gamma,Z_p,\Omega_p)$. We can always reduce to the case of a framed variation by passing to a cover of $M$, or by restricting to a contractible open subset $U\subset M$.
A framed variation of BPS structures  $(\Gamma,Z_p,\Omega_p)$ over a manifold $M$ gives rise to a holomorphic map
\begin{equation}\label{period}\pi \colon M\to \Hom_\bZ(\Gamma,\bC)\isom \bC^n, \quad p\mapsto Z_p,\end{equation}
which we call the period map. 
The variation  will be called miniversal if this map is a 
  local isomorphism. A general variation  will be called miniversal if the framed variations obtained by restricting  to  small open subsets of $M$ are miniversal.

 \subsection{Behaviour of BPS invariants}
 
Condition (V4) in Definition \ref{nick} completely describes the variation of the BPS (or equivalently DT)  invariants: if one knows the $\Omega_p(\gamma)$ at some point $p\in M$ then  they are uniquely determined at all other points $p\in M$.   This statement follows immediately from the following result.

\begin{lemma}
\label{cod}
Given a BPS structure $(\Gamma,Z,\Omega)$, the element \[\bS(\Delta)^*=\prod_{\ell\subset \Delta} \;\bS(\ell)^* \in \Hat{G}_{\Delta}\]determines
the invariants $\Omega(\gamma)$ for all classes $\gamma\in \Gamma$ satisfying $Z(\gamma)\in \Delta$.
\end{lemma}

\begin{proof}
Suppose that  $\Delta$ is the disjoint union of two subsectors $\Delta=\Delta_+\sqcup \Delta_-$ where we make the convention that $\Delta_+$ lies in the anti-clockwise direction. There is an obvious  decomposition
\[\Hat{\fg}_{\Delta}=\Hat{\fg}_{\Delta_+}\oplus \Hat{\fg}_{\Delta_-}.\]
It follows from this  that each element $g\in \Hat{G}_{\Delta}$ decomposes uniquely as a product $g=g_+\cdot g_-$ with $g_\pm\in \Hat{G}_{\Delta_{\pm}}$.  But the obvious relation
\[\bS(\Delta)^*=\bS(\Delta_+)^*\cdot\bS(\Delta_-)^*,\]
gives an example of such a decomposition, so by uniqueness it follows that $\bS(\Delta)^*$ determines the elements $\bS(\Delta_\pm)^*$. 

The  product \eqref{fin} can be thought of as corresponding  to a decomposition of $\Delta$ into a finite number of subsectors, each containing a single ray of height $<N$. Applying the first part repeatedly therefore shows that $\bS(\Delta)^*_{<N}$  determines each element $\bS(\ell)^*_{<N}$ for rays $\ell\subset \Delta$. Since this holds  for arbitrarily large $N$ it follows that $\bS(\Delta)^*$ determines the elements $\bS(\ell)^*$ for all such rays. But the elements $\bS(\ell)^*$   faithfully encode the invariants $\DT(\gamma)$  for classes $\gamma\in \Gamma$  satisfying $Z(\gamma)\in \ell$, so the result follows.
\end{proof}

\begin{remark}
\label{abelian}
Suppose given a framed variation of BPS structures $(\Gamma,Z_p,\Omega_p)$ over a manifold $M$ such that at some point  (and hence at all points) $p\in M$ the form $\<-,-\>=0$ vanishes.  Then  the Lie algebras $\Hat{\fg}_\Delta$ are abelian, and hence so too are the associated groups $\Hat{G}_\Delta$. It follows that the wall-crossing formula is satisfied if all elements $\bS_p(\ell)^*$ are taken to be constant as $p\in M$ varies. It follows from Lemma \ref{cod} that  the BPS invariants $\Omega_p(\gamma)=\Omega(\gamma)$ are constant.
 \end{remark}

One can also easily see the following statement: suppose given a framed variation of BPS structures $(\Gamma,Z_p,\Omega_p)$ over a manifold $M$. Then for any fixed class $\gamma\in \Gamma$ there is a locally-finite collection of codimension one real  submanifolds $\cW_i\subset M$ such that the invariant $\Omega(\gamma)$ is constant on the open complement \begin{equation}
\label{comp}M\setminus \bigcup \cW_i.\end{equation}
The submanifolds $\cW_i$ are called walls for the class $\gamma$, and the connected components of the complement \eqref{comp} are called chambers.


\section{Convergent BPS structures}
\label{converge}

 The aim of this section is to prove the claims made in Section \ref{ref} that enable one to view the BPS automorphisms of a convergent BPS structure as being partially-defined automorphisms of the twisted torus. 

\subsection{Convergent BPS structures}

Recall  from Section \ref{defns} that a BPS structure $(\Gamma,Z,\Omega)$ is called convergent if there exists $R>0$ such that 
\begin{equation}
\label{blockhead}\sum_{\gamma\in \Gamma} |\Omega(\gamma)| \cdot e^{-R |Z(\gamma)|}<\infty.\end{equation}
Let us fix an arbitrary norm
$\|\cdot\|$  on the finite-dimensional vector space $\Gamma\tensor_\bZ \bR$. Then we can find $k_1<k_2$ such that for any active class $\gamma\in \Gamma$  there are inequalities
\begin{equation}\label{ein}k_1\cdot \|\gamma\| < |Z(\gamma)|< k_2\cdot \|\gamma\|.\end{equation}
The first inequality is the support property and requires the assumption that $\gamma\in \Gamma$ is active, whereas the second is just the fact that the linear map $Z\colon \Gamma\tensor_\bZ \bR\to \bC$ has bounded norm. It follows that in  expressions such as  \eqref{blockhead}, which are sums over active classes, we can equally well use $\|\gamma\|$ or $|Z(\gamma)|$ in the exponential factor.

\begin{lemma}
\label{sven}
Let $(\Gamma,Z,\Omega)$  be a convergent BPS structure and choose a norm $\|\cdot\|$  as above.   Then for any $\epsilon>0$ there is an $R>0$ such that
\begin{equation}\label{andrew}\sum_{\gamma\in \Gamma} \|\gamma\|^p \cdot |\DT(\gamma)|\cdot e^{-R  \|\gamma\|}<\epsilon,\end{equation}
for each integer $p\in \{0,1,2\}$. 
\end{lemma}

\begin{proof}
 Clearly it is enough to prove the result for each $p\in \{0,1,2\}$ separately, so we can consider $p$ fixed. For  $x\gg 0$ there is an inequality
\[\frac{x^p}{e^{2x}-1}<e^{-x}.\]
By the discussion above, we can take $S>0$ so that the inequality \eqref{blockhead} holds with $R=S$ and $|Z(\gamma)|$ replaced with $\|\gamma\|$ in the exponential factor.
The   numbers $\|\gamma\|$  are bounded below so increasing $S$ if necessary we can  assume that
\[S^p\|\gamma\|^p \cdot (e^{2S\|\gamma\|}-1)^{-1}< e^{-S\|\gamma\|}<1,\]
for all  classes $\gamma\in \Gamma$.
Taking $R=2S$ and using the definition of DT invariants \eqref{bps} gives
\[\sum_{\gamma\in \Gamma}\|\gamma\|^p \cdot |\DT(\gamma)|\cdot e^{-R \|\gamma\|}=\sum_{\gamma\in \Gamma} \sum_{n\geq 1}  \frac{1}{n^2}\cdot \|n\gamma\|^p \cdot |\Omega(\gamma)|\cdot e^{-2nS\|\gamma\|}\]\[
\leq \sum_{\gamma\in \Gamma}  \|\gamma\|^p\cdot |\Omega(\gamma)|\cdot \frac{e^{-2S\|\gamma\|}}{1-e^{-2S\|\gamma\|}}<S^{-p}\cdot \sum_{\gamma\in \Gamma} |\Omega(\gamma)| \cdot e^{-S \|\gamma\|} <\infty.\]
Since the numbers $\|\gamma\|$ appearing in the exponential factor are bounded below, by increasing $R$ we can  ensure that this sum is smaller than any given $\epsilon>0$. 
\end{proof}

\subsection{Partially defined automorphisms}\label{tird}
Let $(\Gamma,Z,\Omega)$ be a  BPS structure, and fix a  convex sector $\Delta\subset \bC^*$. As above, we also fix  a norm $\|\cdot \|$ on the vector space $\Gamma\tensor_\bZ \bR$.
In the next subsection we will be interested in controlling Hamiltonian flows of functions on the twisted torus $\bT$  of the form $\DT(\gamma)\cdot x_\gamma$, for those classes $\gamma\in \Gamma$ satisfying $Z(\gamma)\in \Delta$. Thus it makes sense to define,  
for each real number $R>0$, a subset 
\[V_\Delta(R)=\big\{\xi\in \bT: \gamma\in \Gamma \text{ active with }  Z(\gamma)\in \Delta\implies |x_\gamma(\xi)|<\exp(-R \|\gamma\|)\big\}.\]
It is not clear that the subset $V_\Delta(R)\subset \bT$ is open in general, so we define $U_\Delta(R)$ to be its interior.
We note the obvious implications
\[R_2\geq R_1 \implies U_{\Delta}(R_2)\subseteq U_{\Delta}(R_1), \quad \Delta_2\supseteq \Delta_1 \implies U_{\Delta_2}(R)\subseteq U_{\Delta_1}(R).\]
We think of the open subsets $U_\Delta(R)\subset \bT$ as forming a system of neighbourhoods of   the boundary in some fictitious partial compactification of $\bT$.

\begin{lemma}
\label{dyl}
The open subset $U_\Delta(R)=\interior V_\Delta(R)\subset \bT$ is non-empty.
\end{lemma}

\begin{proof}
Take an element $z\in \Delta$.  Since $\Delta$ is acute we can find a constant $k>0$ such that
\[Z(\gamma)\in \Delta \implies \Re \big(Z(\gamma)/z\big) > k\cdot |Z(\gamma)|.\]
Take a point $\xi\in \bT$. Using the support property, we conclude that for any active class $\gamma\in \Gamma$ with $Z(\gamma)\in \Delta$, and all real numbers $S>0$, there are inequalities
\[|e^{-SZ(\gamma)/z}\cdot \xi(\gamma)| =e^{-S\Re (Z(\gamma)/z)}\cdot |\xi(\gamma)|<e^{-kS|Z(\gamma)|}\cdot |\xi(\gamma)|<e^{-cS\|\gamma\|}\cdot |\xi(\gamma)|,\]
for some universal constant $c>0$. Since the numbers $\|\gamma\|$ are bounded below, it follows that the element $e^{-SZ/z}\cdot \xi\in \bT$ lies in the subset $V_\Delta(R)$ for sufficiently large $S>0$. The  argument applies uniformly to all elements $\xi$ lying in a compact subset of $\bT$, so we conclude that $V_\Delta(R)$ has non-empty interior.
\end{proof}

Let us define a $\Delta$-map germ to be an equivalence class of holomorphic maps of the form $f\colon U_\Delta(R)\to \bT$, where two such maps $f_i\colon U_\Delta(R_i)\to \bT$  are considered to be equivalent if they agree on the intersection $U_\Delta(\max(R_1,R_2))$ of their domains. Thus a $\Delta$-map germ gives a partially-defined holomorphic map  $f\colon \bT\dashrightarrow \bT$ which is well-defined on  the open subset $U_\Delta(R)\subset \bT$ for all sufficiently large $R> 0$. A $\Delta$-map germ $f$ will be called bounded if for any $\epsilon>0$ one has
\begin{equation}\label{finger}f(U_\Delta(R+\epsilon))\subset U_\Delta(R)\end{equation}
for all sufficiently large $R>0$. 

It is clear  from \eqref{finger} that composition of such maps is well-defined and hence gives the set of  bounded $\Delta$-map germs the structure of a monoid. 
 We write  $\Aut_\Delta(\bT)$ for the group of invertible elements in this monoid, and call the elements of this group invertible bounded $\Delta$-map germs.  Elements of  $\Aut_\Delta(\bT)$ give  partially-defined holomorphic maps $f\colon \bT\dashrightarrow \bT$  with partially-defined inverses, 
 and hence give a precise meaning to the idea of a partially-defined automorphism of $\bT$.

\subsection{BPS automorphisms}
Let $(\Gamma,Z,\Omega)$ be a convergent BPS structure and fix  a  convex sector $\Delta\subset \bC^*$. Recall the definition of the formal series $\DT(\ell)$ associated to a ray  $\ell\subset \bC^*$ from Section \ref{raydiagram}. In this section we show that this series is absolutely convergent on suitable open subsets of $\bT$ and that the Hamiltonian flow of the resulting holomorphic function gives a partially-defined automorphism of $\bT$ in the precise sense described in the last subsection.

\begin{prop}
\label{prop}
For all sufficiently large $R>0$ the following statements hold.
For each ray $\ell\subset \Delta$, the power series $\DT(\ell)$ is absolutely convergent on $U_\Delta(R)$, and hence defines a holomorphic function \[\DT(\ell)\colon U_\Delta(R)\to \bC.\]
Moreover the time 1 Hamiltonian flow of this function gives a holomorphic map
\[\bS(\ell)\colon U_\Delta(R)\to \bT\]
which defines an invertible, bounded  $\Delta$-map  germ.
\end{prop}

\begin{proof}
Take assumptions as in the statement and choose $\epsilon>0$. For convenience we choose the norm $\|\cdot \|$ on $\Gamma_\bZ \tensor \bR$ so that for all $\beta,\gamma\in \Gamma$ there is an inequality
\begin{equation}
\label{blbly}|\<\beta,\gamma\>|< \|\beta\|\cdot \|\gamma\|.\end{equation}
Take $R>0$  satisfying the conclusion of Lemma \ref{sven}.  Thus for each ray $\ell\subset \Delta$ we can choose  $M(\ell)>0$ such that
\begin{equation}
\label{thatwas}\sum_{Z(\gamma)\in \ell} \|\gamma\|^p \cdot |\DT(\gamma)|\cdot e^{-R \|\gamma\|}<M(\ell),\end{equation}
for each $p\in \{0,1,2\}$, and such that
\begin{equation}
\label{flfl}\sum_{\ell\subset \Delta} M(\ell)<\epsilon.\end{equation}

Since the $\|\gamma\|$ are bounded below,  the $p=0$ case of \eqref{thatwas} immediately implies that for each point $\xi\in U_\Delta(R)$
\begin{equation}
\label{again}\sum_{Z(\gamma)\in \ell}| \DT(\gamma)|\cdot  |x_\gamma(\xi)|< \sum_{Z(\gamma)\in \ell}| \DT(\gamma)|\cdot  e^{-R\|\gamma\|}<\infty,\end{equation}
 which proves that $\DT(\ell)$ is absolutely convergent.
 
 The Hamiltonian flow of $\DT(\ell)$ is defined by the differential equation
\begin{equation}
\label{ca}\frac{dx_\beta}{dt} =\{\DT(\ell),x_\beta\}=x_\beta \cdot \sum_{Z(\gamma)\in \ell}\DT(\gamma)\cdot \<\gamma,\beta\>\cdot x_\gamma.\end{equation}
Using \eqref{blbly} it follows
that  providing the flow stays in $U_\Delta(R)$ it satisfies
\begin{equation}
\label{co2}\Big|\frac{d}{dt}\log x_\beta\Big|\leq  \sum_{Z(\gamma)\in \ell} |\DT(\gamma) |\cdot \|\beta\|\cdot  \|\gamma\| \cdot e^{-R \|\gamma\|} <M(\ell)\cdot  \|\beta\|.\end{equation}

Let us consider an integral curve $\phi(t)$ for this flow which starts at some point $\phi(0)\in U_\Delta(R+\epsilon)$. We claim that this curve extends to all times $t\in [0,1]$ and remains in $U_\Delta(R)$. It follows from this that the time 1 Hamiltonian flow $\bS(\ell)$ exists, and defines a holomorphic map
\[\bS(\ell)\colon U_\Delta(R+\epsilon)\to U_\Delta(R).\]
Since we can make $\epsilon>0$ arbitrarily small by increasing $R>0$, the resulting $\Delta$-map germ is  bounded. 
 Considering the opposite flow then shows that the map is invertible, which completes the proof of Proposition \ref{prop}.
 
 To prove the claim, let us define $t_0\in [0,1]$ to be the supremum of $s\in [0,1]$ such that the integral curve $\phi(t)$ can be extended to a map $\phi\colon [0,s]\to U_\Delta(R)$. We obtain an integral curve $\phi\colon [0,t_0)\to U_\Delta(R)$. Equation \eqref{co2} implies that\begin{equation}
\label{curry}e^{-M(\ell) \cdot \|\beta\|}\cdot  |x_\beta(\phi(0))|\leq |x_\beta(\phi(t))|\leq e^{M(\ell) \cdot \|\beta\|}\cdot  |x_\beta(\phi(0))|\end{equation}
for  $0\leq t<t_0$. Thus the flow stays in the compact subset of $\bT$ defined by these inequalities, and hence extends to a flow $\phi\colon [0,t_0]\to \bT$. Applying \eqref{curry} 
to the  active classes $\beta\in \Gamma$ then shows that  $\phi(t_0)\in U_\Delta(R)$. But now local existence of solutions to \eqref{ca} shows that we can extend the flow a little further, which contradicts  the definition of $t_0$ unless $t_0=1$. This proves the claim.\end{proof}

\subsection{Compositions of BPS automorphisms}
Let $(\Gamma,Z,\Omega)$ be a convergent BPS structure and fix  a  convex sector $\Delta\subset \bC^*$.  Fix $\epsilon>0$ and take $R>0$ as in Lemma \ref{sven}. Recall the  notion of the height of  a ray from Section \ref{raydiagram}. For each ray $\ell\subset \Delta$ define $M(\ell)>0$ as in the proof of Proposition \ref{prop} so that \eqref{thatwas} and \eqref{flfl} hold. Note that if a ray has height $> H$ then by the second inequality in \eqref{ein} we have  $\|\gamma\|>H/k_2$ for all nonzero terms in the sum \eqref{thatwas}. It follows that by increasing $R$ we can assume that for all $H>0$
\begin{equation}
\label{wa}\sum_{\height(\ell)> H} M(\ell)<\epsilon\cdot e^{-H}.\end{equation}

The proof of Proposition \ref{prop} shows that for each $S>R$ and each ray $\ell\subset \Delta$ the map $\bS(\ell)$ gives a well-defined embedding $U_\Delta(S+M(\ell))\to U_\Delta(S)$. 
For each $H>0$, the composition in anti-clockwise order of the $\Delta$-map  germs $\bS(\ell)$ corresponding to the finitely many rays $\ell\subset \Delta$ of height $\leq H$ therefore gives a well-defined map $U_\Delta(R+\epsilon)\to U_\Delta(R)$, and hence a  bounded $\Delta$-map germ
\[\bS(\Delta)_{\leq H}=\prod_{\height(\ell)\leq H} \bS(\ell)\in \Aut_\Delta(\bT).\]
The remaining task is to make sense of the limit of this map germ as $H\to \infty$. 
 
Suppose that a $\Delta$-map germ $f$ is well-defined on the open subset $U_\Delta(R)\subset \bT$. We define the $R$-norm of $f$ to be the infimum of the real numbers $K>0$ such that for all $\beta\in \Gamma$ and all $\xi\in U_\Delta(R)$
 \[e^{-K\cdot  \|\beta\|} \,|x_\beta(\xi)|\leq |x_\beta(f(\xi))|\leq e^{K\cdot \|\beta\|}\, |x_\beta(\xi)|.\]
 The proof of Proposition \ref{prop} shows that the $\Delta$-bounded map germ $\bS(\ell)$ is well-defined on $U_\Delta(R+\epsilon)$ and has norm $<M(\ell)$ there.

\begin{lemma}
\label{popo}
Suppose given a sequence  of invertible $\Delta$-bounded map germs $f_n\in G_\Delta(\bT)$, all defined on a fixed 
 $U_\Delta(R)$,  and such that the $R$-norm $d(m,n)$ of the composite $f_n^{-1}\circ f_m$ goes to zero as $\min(m,n)\to \infty$. Then the holomorphic maps $f_n$ have a uniform limit, which is itself an invertible $\Delta$-bounded map germ, and hence defines a limiting element $f_\infty\in G_\Delta(\bT)$.
\end{lemma}

\begin{proof}
By definition, for any $\xi\in U_\Delta(R)$ we have
\[e^{-d\cdot  \|\beta\|} \,|x_\beta(f_m(\xi))|\leq |x_\beta(f_n(\xi))|\leq e^{d\cdot \|\beta\|}\, |x_\beta(f_m(\xi))|,\]
where $d=d(m,n)$. Thus the $x_\beta(f_n(\xi))$ converge uniformly on compact subsets.
\end{proof}

\begin{prop}
\label{prop2}
The $\Delta$-map germs $\bS(\Delta)_{\leq H}$ have a well-defined   limit  \[\bS(\Delta)=\lim_{H\to \infty} \bS(\Delta)_{\leq H} \in \Aut_\Delta(\bT).\]
\end{prop}

\begin{proof}
Let us consider two finite compositions $P_1$ and $P_2$ of the maps $\bS(\ell)$ with $\ell\subset \Delta$ as above, which  differ by the insertion of one extra term. Thus the two products can be written in the form
$P_1=AC$ and $P_2=ABC$ with \[ A=\prod_i \bS(\ell_i), \quad B=\bS(\ell_j), \quad C=\prod_k \bS(\ell_k).\]
The  composite map germ
 $X=P_2 P_1^{-1}$ is well-defined on $U_\Delta(R+2\epsilon)$. We claim that it  has norm $<2M(\ell_j)$ there.
 Given  heights $H_2>H_1>0$ we can apply the claim repeatedly and use \eqref{wa} to deduce that
\[\bS(\Delta)_{\leq H_2} = \bS(\Delta)_{\leq H_1}\circ X(H_1,H_2),\]
where $X(H_1,H_2)$ has norm at most  $2\epsilon\cdot e^{-H_1}$. The result then follows from  Lemma \ref{popo}.

To prove the claim, note first that $X=ABA^{-1}$ is also the time 1 flow of a vector field on $\bT$, namely the  Hamiltonian vector field of the function $\DT(\ell_j)\circ A^{-1}$ pushed forward by the map $A$. We can therefore apply the same argument as in Proposition \ref{prop} provided we can bound the norm of this vector field as in  \eqref{co2}. To do this it is enough to bound the norm of the difference between the derivative of the map $A$ and the identity at points of $U_\Delta(R)$. This in turn is achieved by the same argument as Proposition \ref{prop} using the $p=2$ case of inequality \eqref{thatwas}. We leave the details to the reader.\end{proof}

\subsection{Birational transformations}
\label{birat}

The following example was first observed by Kontsevich and Soibelman \cite[Section 2.5]{KS1}.
Suppose that a ray $\ell$ contains a single active class $\gamma$, and that moreover $\Omega(\gamma)=1$.  The series $\DT(\ell)$ is then the dilogarithm
\[\DT(\ell)=-\sum_{n\geq 1} \frac{x_{n\gamma}}{n^2}=-\sum_{n\geq 1} \frac{x_{\gamma}^n}{n^2}.\]
This converges absolutely for $|x_\gamma|<1$ and hence defines a holomorphic function on the open subset \[U(R)=\{f\in \bT: |f(\gamma)|<e^{-R}\}\subset \bT,\]
for any $R>0$. The argument of Proposition \ref{prop} shows that the associated time 1 Hamiltonian flow $\bS(\ell)$ is well-defined on $U(R)$, and indeed we can directly compute
\[\bS(\ell)^*(x_\beta)=\exp \Big\{-\sum_{n\geq 1} \frac{x_{\gamma}^n}{n^2}, -\Big\} (x_\beta)=x_\beta\cdot  \exp\Big(  -\sum_{n\geq 1} \frac{x_{\gamma}^n}{n^2}\cdot \<n\gamma,\beta\>\Big)\]\[
= x_\beta\cdot \exp\Big( \<\gamma,\beta\>  \log(1-x_\gamma)\Big)=x_\beta\cdot (1-x_\gamma)^{\<\gamma,\beta\>}.\]
Thus the map $\bS(\ell)$  extends holomorphically to the Zariski open subset of $\bT$ which is the complement of the divisor $x_\gamma=1$. It therefore defines a birational automorphism of $\bT$.

\begin{prop}
\label{barney}
Suppose that $(\Gamma,Z,\Omega)$ is a  generic, integral and ray-finite BPS structure. Then for any ray $\ell\subset \bC^*$ the   BPS automorphism $\bS(\ell)$ extends to a birational automorphism of $\bT$, whose action on twisted characters is given by 
\begin{equation}\bS(\ell)^*(x_\beta)=x_\beta\cdot \prod_{Z(\gamma)\in \ell}(1-x_\gamma)^{\,\Omega(\gamma)\cdot\<\gamma,\beta\>}.\end{equation}
\end{prop}

\begin{proof}
The argument of Proposition \ref{prop} shows that the automorphism $\bS(\ell)$ is well-defined on a suitable open subset of $\bT$. Under the generic assumption  
 the Hamiltonian flows of the characters appearing in the function $\DT(\ell)$ commute, so we can compute as before
\[\bS(\ell)^*(x_\beta)=\exp \Big\{-\sum_{Z(\gamma)\in \ell} \sum_{n\geq 1} \Omega(\gamma)\cdot \frac{x_{\gamma}^n}{n^2}, -\Big\} (x_\beta)=x_\beta\cdot  \exp\Big( - \sum_{n\geq 1}\sum_{Z(\gamma)\in \ell}  \Omega(\gamma)\cdot \frac{x_{\gamma}^n}{n^2}\cdot \<n\gamma,\beta\>\Big)\]\[
= x_\beta\cdot \exp\Big(\sum_{Z(\gamma)\in \ell}  \<\gamma,\beta\>  \Omega(\gamma) \log(1-x_\gamma)\Big)=x_\beta\cdot \prod_{Z(\gamma)\in \ell}(1-x_\gamma)^{\,\Omega(\gamma)\cdot\<\gamma,\beta\>}.\]
When the invariants $\Omega(\gamma)$ are integers  this is clearly a birational automorphism of $\bT$. \end{proof}

\end{appendix}

\end{document}